\documentclass[final]{siamonline190516}

\usepackage[utf8x]{inputenc}

 \usepackage[T1]{fontenc}

\usepackage{amsfonts}
\usepackage{graphicx}
\usepackage{epstopdf}
\usepackage{algorithmic}

\usepackage{amsmath}            %
\usepackage{amssymb}            %
\usepackage{mathrsfs}           %
\usepackage{mathtools}          %

\usepackage{color}              %
\usepackage{tikz}
\usetikzlibrary{positioning}
\usepackage{pgfplots}
\usepackage{pgfplotstable}
\usepgfplotslibrary{colorbrewer}
\usepgfplotslibrary{groupplots}
  \pgfplotsset{compat=newest}
\usepackage{xcolor}
	\definecolor{bluecolor}{RGB}{122,166,218}
	\definecolor{redcolor}{RGB}{213, 78, 83}
	\definecolor{orangecolor}{RGB}{231, 140, 69}
	\definecolor{greencolor}{RGB}{185, 202, 74}
	\definecolor{purplecolor}{RGB}{195, 151, 216}

\usepackage{cite}

\newsiamremark{example}{Example}
\newsiamremark{problem}{Problem}
\crefname{problem}{Problem}{Problems}
\allowdisplaybreaks
\newcommand{\bs}[1]{{\boldsymbol #1}}
\newcommand{\D}{{\mathcal D}}

\newcommand{\VMC}{\operatorname{V}_{\!\mathrm{MC}}}
\newcommand{\VSRD}{\operatorname{V}_{\!\mathrm{SRD}}}

\title{Optimal control under uncertainty with joint chance state constraints: almost-everywhere
bounds, variance reduction, and application to\\ (bi-)linear elliptic PDEs%
\thanks{Version from \today.\funding{R.H.\ acknowledges support by the Deutsche Forschungsgemeinschaft within CRC 154 (project B04) and by the PGMO program funded by Electricité de France. G.S.\ and F.W.\ were partially supported by the Simons Foundation/SFARI (560651, AB). G.S.\ also acknowledges support by the U.S.~Department of Energy, Office of Science Energy Earthshot Initiative under Award \#DE-SC0024721.}}}

\author{Rene Henrion\thanks{Weierstrass Institute for Applied Analysis
    and Stochastics, Berlin, Germany, \email{henrion@wias-berlin.de}}
  \and
  Georg Stadler%
  \thanks{Courant Institute of Mathematical Sciences, New York
    University, New York, USA, \email{stadler@cims.nyu.edu}, \email{fw18@nyu.edu}}
  \and Florian Wechsung\footnotemark[3]
}

\begin{document}

\maketitle

\begin{abstract}
  We study optimal control of PDEs under uncertainty with the state variable subject to joint chance constraints. The controls are deterministic, but the states are probabilistic due to random variables in the governing equation. Joint chance constraints ensure that the random state variable meets pointwise bounds with high probability. For linear governing PDEs and elliptically distributed random parameters, we prove existence and uniqueness results for almost-everywhere state bounds.
  Using the spherical-radial decomposition (SRD) of the uncertain variable, we prove that when the probability is very large or small, the resulting Monte Carlo estimator for the chance constraint probability exhibits substantially reduced variance compared to the standard Monte Carlo estimator. We further illustrate how the SRD can be leveraged to efficiently compute derivatives of the probability function, and discuss different expansions of the uncertain variable in the governing equation. Numerical examples for linear and bilinear PDEs compare the performance of Monte Carlo and quasi-Monte Carlo sampling methods, examining probability estimation convergence as the number of samples increases. We also study how the accuracy of the probabilities depends on the truncation of the random variable expansion, and numerically illustrate the variance reduction of the SRD.
\end{abstract}



\tableofcontents

\section{Introduction}
\label{sec:intro}

While deterministic optimization and control problems governed by
partial differential equations (PDEs) are a well establish research
area \cite{BorziSchulz12, HinzePinnauUlbrichEtAl09, Troltzsch10}, only
over the last one or two decades researchers have started to systematically study
these problems in the presence of uncertainty, e.g., \cite{
  KouriHeinkenschlossRidzalEtAl13, %
  GarreisSurowiecUlbrich21, %
  ChenGhattas21,  %
  AlexanderianPetraStadlerEtAl17,  %
  KouriShapiro18,    %
  MilzUlbrich24       %
}.
Optimization under uncertainty introduces new theoretical and algorithmic difficulties because uncertainty can substantially increase the dimensionality of the problem. This adds complexity to optimization problems governed by PDEs, as it necessitates random space approximation and one requires adapted objectives and constraints.
Here we focus
on probabilistic constraints on the state in optimal control
problems governed by elliptic equations. In particular, we consider joint chance state constraints which require that
realizations of the distribution of states must satisfy pointwise
bound constraints with a certain probability. A
precise formulation is given next.

\subsection{Problem statement}\label{subsec:problem}
For a bounded, sufficiently regular domain $\D\subset \mathbb R^d$, typically
$d\in\{1,2,3\}$, we consider an optimal control problem under
uncertainty, i.e.:
\begin{equation}\label{eq:optcon}
  \min_{u\in U_{\!\text{ad}},\, y\in X} \mathcal J(u,y),
\end{equation}
where $U_{\!\text{ad}}$
is a convex closed set of admissible controls, and $X$ the function space of the random states. The uncertainty enters through the equation that relates the control $u$ and the state $y$. We use a probability space $(\Omega,\mathcal A,\mathbb P)$, i.e., $\omega\in \Omega$ are events, $\mathcal A$ is a sigma-algebra on $\Omega$ and $\mathbb P$ is a probability measure. The governing equation is an elliptic PDE with random variable $\xi$ defined in $(\Omega,\mathcal A,\mathbb P)$ and taking values
in a Hilbert space $\mathscr H$. The state $y(\omega)$ and the control $u$ satisfy the equation
\begin{equation}\label{eq:redstate}
A(u) y(\omega) + B \xi(\omega) + Cu  = f,
\end{equation}
where for $u\in U_{\!\text{ad}}$, $A(u):Y\mapsto
H^{-1}(\D)$ with $Y$ a subspace of $H^1(\D)$ is a linear invertible elliptic operator that is
differentiable with respect to $u$, $B:\mathscr
H\mapsto H^{-1}(\D)$ and $C:U_{\!\text{ad}}\to L^2(\D)$ are bounded linear operators,
and $f\in L^2(\D)$. 

The equation form \eqref{eq:redstate} includes
linear and bilinear PDEs, and we will give concrete examples later in
this introduction.
Additionally to \cref{eq:optcon} and \cref{eq:redstate}, for $p\in
[0,1]$ we consider the following joint chance constraints for the
state variable,
\begin{equation}\label{eq:chance}
  \mathbb P(\omega\in \Omega \,| \, \underline y( \bs x) \le y( \bs x, \omega)\le \bar y(\bs x)
  \text{ for a.\,a.\ } \bs x\in \D) \ge p,
\end{equation}
where $\underline y, \bar y\in L^2(\D)$ are given lower and upper bounds, respectively. 
The constraint \cref{eq:chance} states that
realizations of the state, which is a random variable, must be between
$\underline y$ and $\bar y$ in a pointwise sense with probability of
at least $p$.
In this paper, we develop efficient methods to approximate and solve \eqref{eq:optcon}, \eqref{eq:redstate}, and \eqref{eq:chance} numerically.

\subsection{Related literature}
Chance constraints were introduced in the 1950s \cite{charnes1958}, and fundamental contributions to their theoretical and numerical foundations were later made by Pr{\'e}kopa \cite{prekopa1995}; see \cite{Shapiro2009} for a more modern presentation.
Recently, these probabilistic constraints have been considered in the context of infinite-dimensional optimization, in particular
PDE-constrained optimization \cite{Schuster2022, FarshbafShakerGugatHeitschEtAl20, FarshbafShakerHenrionHoemberg18, geletu2020, Heitsch2024, Perez2022, Teka23, vanAckooij2024, Kouri2023, geiersbach2023optimality, ChenGhattas21, HantouteHenrionPerez19}. These works focus on
applications (for example, optimization of gas transport \cite{Heitsch2024,Schuster2022}, aeronautics \cite{caillau2018} or population growth \cite{Perez2022}), numerical approaches (for example, \cite{Kouri2023, geletu2020, ChenGhattas21}) or theoretical aspects (for example, existence and stability of solutions \cite{FarshbafShakerHenrionHoemberg18}, convergence of algorithms \cite{Teka23}, optimality conditions \cite{vanAckooij2024,geiersbach2023optimality} or explicit estimates for subdifferentials \cite{HantouteHenrionPerez19}).

Numerical approaches for chance constraints can roughly be classified in (1) methods that relax the constraints, resulting in formulations with nicer properties (e.g., convexity, differentiability) and (2) methods that approximate the original chance constraints. 
One method falling into the first class is using the Conditional Value-at-Risk, but other relaxations are possible (e.g., \cite{Nemirovski2006,Kouri2023}). %
Facilitating solution methods through relaxation of the chance constraint 
comes at the cost of potentially significant conservatism of the resulting solutions.
The second class of approaches approximate the original chance constraints. This approximation could be discrete (as in sample average approximation \cite{Pagnoncelli2009}) or smooth (e.g., \cite{pena,geletu17,caillau2018,VanAckooijHenrion14}). Consequently, solution algorithms from discrete or non-linear optimization are used. In this paper, we follow an approach based on the {\it spherical-radial decomposition} (SRD) of elliptically distributed (e.g., Gaussian) random vectors, which has been shown to be successful both in the theoretical analysis of probability functions \cite{VanAckooijHenrion14,VanAckooijHenrion17,HantouteHenrionPerez19,vanAckooij2024,geiersbach2023optimality} and in numerical solutions of chance-constrained optimization problems \cite{Perez2022,Heitsch2024,FarshbafShakerGugatHeitschEtAl20,BertholdHeitschHenrionEtAl21}.

We emphasize that \eqref{eq:chance} is a {\it joint} chance constraint, that is, it considers the probability over a whole system of random inequalities. Alternatively, one could formulate {\it individual} chance constraints, where the probability is taken over all inequalities individually:
\[
\mathbb P\left(\omega\in \Omega \,| \, \underline y( \bs x) \le y( \bs x, \omega)\le \bar y(\bs x)\right) \ge p\qquad\text{ for a.\,a.\ } \bs x\in \D.
\]
The difference lies in enforcing uniform versus pointwise constraints, %
which must be satisfied with high probability. Individual constraints are less restrictive, allowing for lower optimal objective function values. Furthermore, in specific models (e.g., when control and random parameters are separated as in \eqref{eq:redstate}), individual constraints can be reformulated into explicit deterministic equivalents. In contrast, joint constraints address the need for uniform constraint satisfaction, often reflecting practical requirements. In particular, random states that meet individual bounds with high probability at each point may still fail to meet these bounds uniformly over the entire domain with high probability (see, e.g., \cite[Sec.\ 4.2]{BertholdHeitschHenrionEtAl21}).

Note that in the limit $p\to 1$, the chance constraint \eqref{eq:chance} turns into an ``almost-sure'' constraint. Nevertheless, in general, it is not recommended to formulate an almost-sure constraint as a chance constraint because the latter degenerates in the limit to a constraint violating standard constraint qualifications \cite{geiersbach2023optimality}. We refer to \cite{GeiersbachWollner20, geiersbach2023optimality,GahururuHintermullerSurowiec21}  for theoretical and numerical approaches to treat almost-sure constraints.

\subsection{Contributions, limitations and overview}
In this paper, we make the following \emph{contributions}.
(1) We extend existence, uniqueness, and convexity results to problems with joint chance {\it almost everywhere} state constraints.
(2) We prove theoretically and illustrate numerically that, compared to standard Monte Carlo (MC) sampling, using the spherical-radial decomposition (SRD) substantially reduces the variance for estimation of low, high, or radially close-to-symmetric probabilities. 
(3) Using smoothing properties characteristic of PDE operators, we illustrate the innate dimension reduction for the random state variable and discuss its use in approximating chance constraints.
(4) We present a systematic numerical study of control problems governed by linear and bilinear elliptic problems, comparing standard and SRD-based MC and quasi-MC sampling. We find that a SRD quasi-MC method requires 1-2 orders of magnitude fewer samples than a standard MC method for the same accuracy, which directly translates to a corresponding speedup.

Our approach also has some \emph{limitations}.
(1) For the problems we consider here, the map from the uncertainty to the PDE solution (and thus the constraint) is linear. The methods can be generalized to problems where the constraint function depends nonlinearly on the random variables. However, some advantages of the proposed methods are lost with increasing nonlinearity.
(2) We also assume that for fixed control $u$, the governing equation is linear. Note, however, that we do not require that the control enter linearly, as we illustrate using a bilinear control example.
(3) For the problems considered here, the variance reduction achieved by using SRD MC sampling over standard MC sampling directly translates into computational speed-up. Although variance reduction also applies to problems where the constraint depends nonlinearly on the random variable, the increased computational cost of SRD MC in nonlinear settings may negate the advantages gained from reduced variance.

Finally, we give an \emph{overview}. \Cref{sec:existence-uniqueness} presents basic existence, uniqueness and convexity results for the case that the governing equation \eqref{eq:redstate} is linear in the control $u$. In particular, existing results are generalized to {\it almost everywhere} state constraints. In \cref{sec:spherraddec}, the spherical-radial decomposition is introduced for general elliptically distributed random variables, and reduced variance results are shown for the case of very small, large or close-to-symmetric sets. Furthermore, we provide derivatives of the probability function. In \cref{sec:application}, we apply the general framework specifically to the control problem \eqref{eq:optcon}, discuss truncation of the random variable expansion, and compute explicit gradient formulas of the probability. Finally, in \cref{sec:example-linear,sec:example-bilinear}, we perform comprehensive numerical experiments for a linear and a bilinear control problem, and discuss the results in the context of our theoretical findings.

\subsection{Examples}
Next, we present two examples that fit the setting described in
\cref{subsec:problem}. These examples are studied numerically in
\cref{sec:example-linear,sec:example-bilinear}.
\begin{example}[Tracking-type objective governed by linear PDE with
    uncertain Neumann data]\label{ex:linear}
We split the boundary of the physical domain $\D\subset\mathbb R^d$
into disjoint open sets $\partial \D_1, \partial \D_2$, and assume
given a Gaussian law $d\mu$ with realizations $\xi(\omega) \in
L^2(\partial\D_2)$. We consider the risk-neutral tracking-type optimal control problem 
\begin{equation}\label{eq:ex1-J}
        \begin{aligned}
        & \underset{\mathclap{u\in L^2(\D), \,y\in X}}{\text{minimize}}\:
        & & %
        \frac 12 \int_\Omega
  \int_\D (y(\omega) - y_d)^2 \,d\bs x\,d\mu
  + \frac\alpha 2 \int_\D u^2\,d\bs x %
        \end{aligned}
\end{equation}
subject to a governing equation with uncertain Neumann data on $\partial\D_2$
\begin{subequations}\label{eq:state_linear}
\begin{alignat}{2}
  -\Delta y(\omega) &=
  f + u \qquad &&\text{in } \D, \label{eq:state_linear1}\\
  y(\omega)  &= 0  &&\text{in }\partial\D_1, \label{eq:state_linear2}\\
  \nabla y(\omega)\cdot{\bs n} &= \xi(\omega) &&\text{in
  }\partial\D_2, \label{eq:state_linear3}
  \end{alignat}
\end{subequations}
and joint state chance constraints \eqref{eq:chance}. Here, $y_d,f\in
L^2(\D)$, $\alpha>0$ and $\bs n$ is the unit boundary normal and $X$ is the (Bochner) space of solutions of \eqref{eq:state_linear}. In
\cref{sec:example-linear}, we show that the integration over $\Omega$ in \eqref{eq:ex1-J} can be done analytically, and study methods to approximate the joint chance constraint and the effect of the
bound $p$ on the controls $u$ that minimize \eqref{eq:ex1-J}.
\end{example}

\begin{example}[Bilinear control with uncertain right hand side]\label{ex:bilinear}
  In bilinear control, the governing equation contains a term in which the control $u$ multiplies the state $y$. This structure appears, for example, in the optimal control of quantum systems \cite{BorziSalomonVolkwein08,
    BorziCiaramellaSprengel17, ItoKunisch07}. We assume an
  uncertain right hand side with realizations $\xi(\omega)\in L^2(\D)$ resulting in the governing equation
  \begin{subequations}\label{eq:state_bilin_2}
    \begin{alignat}{2}
      -\Delta y(\omega) + uy(\omega) &=
      f + \xi(\omega) \qquad &&\text{in } \D,\\
      y(\omega)  &= 0  &&\text{in }\partial\D,
    \end{alignat}
  \end{subequations}
  where $f\in H^{-1}(\D)$.
  We consider an objective that only involves the control cost, i.e., the optimization under uncertainty problem with joint chance state constraints is
  \begin{alignat}{2}
    &\underset{\mathclap{u\in U_{\!\text{ad}},\, y\in X}}{\text{minimize}}\:\:\:
    & & %
        \frac 1 2 \int_\D (u-u_0)^2\,d \bs x,
  \end{alignat}
  subject to $y$ satisfies \eqref{eq:state_bilin_2} and the joint chance constraint \eqref{eq:chance}.
  Here, $u_0\in L^2(\D)$ is a given desired control and if $U_{\!\text{ad}}$ includes for instance the pointwise bound $u\ge 0$, invertability of the PDE operator in \eqref{eq:state_bilin_2} is guaranteed. Due to the bilinear equation, this is \emph{not} a linear-quadratic problem. However, for fixed $u$, the map from the uncertain parameter $\xi$ to the
  state variable $y$ is linear.
\end{example}

\subsection{Notation glossary}
We denote the spatial domain by $\D$ and use $\Omega$ for the random space. To distinguish vectors from scalars or scalar functions, we use a bolt font. We commonly abbreviate the spherical-radial decomposition by SRD and Monte Carlo by MC.
The number of unknowns to discretize $\D$ is denoted by $n$. When we use Karhunen-Loeve expansions to describe random variables, we use $K$ for the number of expansion terms, and the index $k$. To approximate the random space $\Omega$ we denote by $N$ the number of MC samples and use the index $i$ for these samples. By $M$ we denote the number of points in $\D$, where the state constraints are evaluated, and we will use the index $j$ for this point set.

\section{Solution existence, uniqueness and convexity for linear governing equation}\label{sec:existence-uniqueness}
In this section, we consider the case where the PDE operator $A$ in
\eqref{eq:redstate} does not depend on the control, i.e., the governing equation is %
\begin{equation}\label{eq:state-lin}
  Ay(\omega) + B\xi(\omega) + Cu = f.
\end{equation}
Since $A$ is assumed to be invertible,
we can consider the (random) state variable $y$ as a function of the (deterministic) control $u$ and the (random)
variable $\xi$:
\begin{equation}\label{eq:controltostate}
y^u(\omega )=A^{-1}(f-Cu-B\xi (\omega)).
\end{equation}
We assume that the random variable $\xi$ is defined as a linear combination of finitely many functions $\xi_0, e_k\in L^2(\D)$ $(k=0,\ldots, K)$:
\begin{equation}\label{afflinrand}
    \xi (\omega ):= \xi_0 + \sum_{k=1}^K\zeta_k(\omega)
    e_{k},
\end{equation}
where $\bs \zeta=(\zeta_1,\ldots,\zeta_K)$ is a $K$-dimensional random vector on the probability space $(\Omega ,\mathcal A, \mathbb P)$.

Assume further that $U_{\!\text{ad}}\subset L^2(\D)$, and consider the probability function
\begin{equation}\label{probfunc}
\varphi \left( u\right) :=\mathbb{P}\left( \omega \in \Omega
:\underline{y}(\bs x)\leq y^u(\bs x, \omega) \leq \overline{y}(\bs x)
\text{ for a.\,a.\ } \bs x\in \mathcal{D}\right)
\end{equation}
and the associated set of controls that satisfy the chance constraint \eqref{eq:chance}, i.e.,
\begin{equation}\label{feasprob}
{U}_{\!\mathrm{pr}}:=\left\{ u\in L^2(\D):\varphi \left( u\right) \geq
p\right\}.
\end{equation}
The next proposition summarizes the properties of $\varphi(\cdot)$ and $U_{\!\text{pr}}$. 
\begin{theorem}\label{structprop}
The probability function \eqref{probfunc} is well-defined and weakly sequentially upper semicontinuous. As a consequence, the feasible set $U_{\!\mathrm{pr}}$ defined in \eqref{feasprob} is weakly sequentially closed for arbitrary $p\in [0,1]$. If the random vector $\bs \zeta$ has a log-concave density (e.g., Gaussian and many others, see \cite[Sec.\ 4.4]{prekopa1995}), then
${U}_{\!\mathrm{pr}}$ is also convex for arbitrary $p\in [0,1]$. 
\end{theorem}
\begin{proof}
  To show that $\varphi $ is well-defined, we verify that for any fixed control $u$ and corresponding state $y^u$, the event set
  \begin{equation}\label{eq:set1}
    \left\{ \omega \in \Omega : \underline{y}(\bs x)\leq y^u(\bs x,\omega)
    \leq \overline{y}(\bs x)\text{ for a.\ a.\ }\bs x\in\D\right\}
  \end{equation}
    belongs to the $\sigma$-algebra $\mathcal{A}$ of the probability space $\left(\Omega, \mathcal{A},\mathbb{P}\right)$. 
To this aim, we introduce the continuous linear solution operator
$S:L^2(\mathcal D)\times\mathbb R^K\to Y$ as
\[
S(u,\bs z):=A^{-1}\left( Cu+\sum_{k=1}^Kz_kBe_k\right)
\qquad \left(u\in L^2(\mathcal D),\,\bs z:=(z_1,\ldots,z_K)\in\mathbb R^K\right)
\]
and define the convex and closed set
\[
\hat{K}:=\{\tilde{y}\in H^1(\mathcal D):A^{-1}(f-B\xi_0)-\overline{y}(\bs x)\leq \tilde{y}(\bs x)
        \leq A^{-1}(f-B\xi_0)-\underline{y} (\bs x)\text{ for a.~a.\ }\bs x\in\D\}.
\]
The closedness of $\hat{K}$ follows from the fact that converging sequences in $\hat K$ also converge in $L^2(\D)$, and that the set of $L^2$-functions with almost-everywhere bounds is closed 
\cite[Sec.\ 2.5]{Troltzsch10}. 
Then, the event set in \eqref{eq:set1} coincides with the set
\begin{equation}\label{eq:set2}
\{\omega\in\Omega:S(u,\bs \zeta (\omega))\in\hat{K}\}=\bs \zeta^{-1}([
S(u,\cdot )]^{-1}(\hat{K})),
\end{equation}
which belongs to the $\sigma$-algebra $\mathcal A$ because the set $[S(u,\cdot )]^{-1}(\hat{K})$ is closed as the preimage of the closed set $\hat{K}$ under the continuous mapping $S(u,\cdot )$. Therefore, $\varphi$ is well-defined. 

To show the remaining statements, we apply \cref{convexitylemma,wsus} from the appendix. In order to do so, we define the abstract data there out of the concrete data used here: %
\[
U:=L^2(\mathcal D), \,\tilde{Y}:=Y, \,K:=\hat{K}, \,g:=S.
\]
We verify first that the assumptions of \cref{wsus} are satisfied, namely that $K$ is weakly closed and $g$ is weakly sequentially continuous. This is indeed the case because, first, $\hat{K}$ is weakly closed as a closed, convex subset of the Banach space $Y$ and, second, because $S$ is weakly sequentially continuous as a continuous linear operator. Then, by \cref{wsus}, the probability function $\tilde{\varphi}$ defined there is weakly sequentially upper semicontinuous. On the other hand, by the coincidence of \eqref{eq:set1} and \eqref{eq:set2}, the probability function $\tilde{\varphi}$ of \cref{wsus} is identical to the probability function $\varphi$ in \eqref{probfunc} in the concrete setting of the current theorem. Thus, $\varphi$ is weakly sequentially upper semicontinuous. 

Likewise, the assumptions of \cref{convexitylemma} are satisfied when we additionally assume in our theorem that $\bs\zeta$ has a log-concave density. Hence, with the definitions made before, the set
\[
M:=\{u\in U:\mathbb{P}(g(u,\bs\zeta )\in K)\geq p\}
=\{u\in U:\mathbb{P}(S(u,\bs\zeta )\in\hat{K})\geq p\}
\] 
defined in that lemma is convex for arbitrary $p\in [0,1]$. Taking into account the equivalence of \eqref{eq:set2} and \eqref{eq:set1} mentioned above as well as the definition of \eqref{feasprob} via \eqref{probfunc}, we end up at the identity $M={U}_{\!\mathrm{pr}}$, so that ${U}_{\!\mathrm{pr}}$ is convex for any $p\in [0,1]$, as was to be shown.
\end{proof}

Note that in the original formulation \eqref{eq:optcon}, we consider the control $u$ and the state $y$ as independent variables that are connected by the equality constraint
\eqref{eq:redstate}.
The previous result allows us to argue existence and uniqueness for the reduced form of \eqref{eq:optcon}, \eqref{eq:redstate}, \eqref{eq:chance}, in which we consider the state $y=y^u$ as a function of the control $u$ using \eqref{eq:state-lin}:
\begin{equation}\label{optprob}
  \min_{u\in {U}_{\!\mathrm{ad}}\cap {U}_{\!\mathrm{pr}}}\hat{\mathcal
    J}(u), \quad \text{where } \hat{\mathcal J}(u):=\mathcal J(u,y^u).
\end{equation}
\begin{theorem}
    Assume that $\hat{\mathcal J}$ is coercive, convex, and lower semicontinuous, and that ${U}_{\!\mathrm{ad}}\cap {U}_{\!\mathrm{pr}}\neq \emptyset$.
   Then, \eqref{optprob} admits a solution. The same holds if we replace the coercivity of $\hat{\mathcal J}$ by the additional boundedness of ${U}_{\!\mathrm{ad}}$. If, in either of the two settings, $\hat{\mathcal J}$ is strictly convex and $\bs\zeta$ has a log-concave density, then \eqref{optprob} has a unique solution. 
\end{theorem}
\begin{proof}
By convexity and lower semicontinuity, $\hat{\mathcal J}$ is
weakly sequentially lower semicontinuous.  Since ${U}_{\!\mathrm{ad}}$ is closed and convex, it is also weakly sequentially closed. With ${U}_{\!\mathrm{pr}}$ being 
weakly sequentially closed by \cref{structprop}, ${U}_{\!\mathrm{ad}}\cap{U}_{\!\mathrm{pr}}$ is  weakly sequentially closed.  Since a coercive, weakly sequentially lower semicontinuous function
attains its infimum over a nonempty and weakly sequentially closed set, there exists a solution to \eqref{optprob}.

Assume now that $\hat{\mathcal J}$ is not necessarily coercive but instead ${U}_{\!\mathrm{ad}}$ is additionally bounded. Then, as a closed, convex, and bounded subset of a reflexive Banach space, it is weakly sequentially compact. As a consequence, the nonempty intersection with the weakly sequentially closed set ${U}_{\!\mathrm{pr}}$ is weakly sequentially compact too.
Now the existence of a solution to problem \eqref{optprob} follows in a classical way from the weak sequential lower semicontinuity of $\hat{\mathcal J}$. 
If $\bs \zeta$ has a log-concave density, then ${U}_{\!\mathrm{pr}}$ is convex by \cref{structprop}. Therefore, ${U}_{\!\mathrm{ad}}\cap{U}_{\!\mathrm{pr}}$ is convex and the result follows from the strict convexity of $\hat{\mathcal J}$.
\end{proof}

\section{Spherical-radial decomposition}\label{sec:spherraddec}
In this section, we revisit the {\it spherical-radial decomposition} (SRD) as a methodology for addressing chance constraints in a slightly more general context than required for our target optimal control problem \cref{eq:optcon}, \cref{eq:redstate} and \cref{eq:chance}. In particular, in \cref{genset,varreduc}, we assume an elliptically distributed random variable (generalizing the Gaussian distribution in our target problem) and that the constraint function is convex (generalizing the linear constraint function needed in \cref{eq:chance}).

The SRD has a somewhat greater level of complexity compared to smoothing methods \cite{caillau2018,Schuster2022,pena} and sample average approximation \cite{Pagnoncelli2009}, which are applicable to any distribution accessible through sampling.
Although the SRD is confined to a specific class of multivariate distributions, including Gaussian distributions, it capitalizes on their intrinsic properties to provide three significant benefits: 
Firstly, it facilitates the analysis of the original problem (e.g., Lipschitz continuity or differentiability of the chance constraint) without necessitating approximations at the outset \cite{VanAckooijHenrion14,VanAckooijHenrion17,HantouteHenrionPerez19}. This approach has been utilized in \cite{vanAckooij2024,geiersbach2023optimality} to derive optimality conditions for a continuous PDE optimal control problem. Secondly, the SRD has reduced variance in probability estimation compared to both direct MC and quasi-MC sampling, which is advantageous for numerically solving chance-constrained optimization problems. Thirdly, the SRD provides explicit gradient expressions for the probability function in both the original and numerically approximated formulations. This is crucial for addressing chance constraints through nonlinear optimization methods. While gradients for numerical approximations are also obtainable with smoothing methods, they cannot be derived from single samples, unlike when using the SRD.
The first aspect is less significant in this context as our focus is on numerical solutions rather than theoretical analysis. The second aspect, concerning variance reduction, has not been systematically explored so far, and we shall examine it more thoroughly in \cref{varreduc}. As for the third aspect, we will derive and employ explicit gradient formulas for the probability functions to solve the optimal control problems outlined in \cref{ex:linear,ex:bilinear}.

\subsection{General setting}\label{genset}
It is well known \cite[eq.\ (2.12)]{fang} that {\it elliptically} distributed $n$-dimensional random vectors $\bs\zeta$ admit a decomposition
\begin{equation}\label{ellipt}
\bs\zeta =\bs m +\tau L\bs\theta ,    
\end{equation}
where $\bs m\in\mathbb{R}^n$, $L\in \mathbb R^{n\times k}$, $k\leq n$, $\tau$ is a one-dimensional non-negative random variable, and $\bs\theta$ is a random vector uniformly distributed on the unit sphere $\mathbb{S}^{k-1}$ of $\mathbb{R}^k$. Moreover, $\tau$ and $\bs\theta$ are independently distributed. With $\tau$ and $\bs\theta$ playing the role of a radial and a spherical variable (in the sense of polar coordinates), respectively, \eqref{ellipt} is referred to as the  spherical-radial decomposition of $\bs \zeta$. A prominent example is a (nondegenerate) Gaussian random vector $\bs \zeta\sim\mathcal{N}(\bs m,\Sigma)$, which yields the decomposition \eqref{ellipt} with $k=n$, $\tau$ having a chi-distribution with $n$ degrees of freedom and $L$ satisfying $LL^T=\Sigma$. Accordingly, the Gaussian probability of some measurable subset $M\subseteq\mathbb{R}^n$ can be represented as
\begin{equation}\label{gaussdec}
\mathbb{P}(\bs \zeta\in M)=\int\limits_{\bs v\in\mathbb{S}^{n-1}}
\mu_\chi (\{r\geq 0:\bs m+rL\bs v\in M\})d\mu_U(\bs v),
\end{equation}
where $\mu_\chi$ is the one-dimensional Chi-distribution with $n$ degrees of freedom and $\mu_U$ is the uniform distribution over 
$\mathbb{S}^{n-1}$.

In this paper, we use the spherical-radial decomposition (SRD) in combination with the inequality $g(u,\bs\zeta )\leq 0$, where $g:U\times\mathbb{R}^n\to\mathbb{R}$, $U$ is a Banach space of controls and $\bs\zeta\sim\mathcal{N}(\bs m,\Sigma)$ is an $n$-dimensional Gaussian random vector. We assume that $g(u,\cdot )$ is convex for all $u\in U$.
Defining the probability function $\varphi:U\to [0,1]$ as
\begin{equation}\label{abstractcc}
\varphi (u):=\mathbb{P}(g(u,\bs \zeta )\leq 0)\qquad (u\in U), 
\end{equation}
\eqref{gaussdec} yields that
\begin{equation*}
\varphi (u)=\int\limits_{\bs v\in\mathbb{S}^{n-1}}
\mu_\chi (\{r\geq 0:g(u,\bs m+rL\bs v\leq 0\})d\mu_U(\bs v)\qquad (u\in U).
\end{equation*}
A considerably more convenient representation can be provided, if the mean $\bs m$ is strictly feasible for $u\in U$: 
\begin{equation}\label{slatercond}
g(u,\bs m)<0.
\end{equation}
In \cite[Prop.~3.11]{VanAckooijHenrion14} it was shown that \eqref{slatercond} holds if
$\varphi (u)\geq 0.5$ and there exists \emph{any} Slater point $\bs z\in\mathbb{R}^n$ (not necessarily $\bs m$), i.e., $g(u,\bs z)<0$. %
Both conditions are mild; the latter is because the target probability level $p$ in a chance constraint is usually close to one.
Under \eqref{slatercond}, the one-dimensional set 
$
\{r\geq 0:g(u,\bs m+rL\bs v)\leq 0\}
$
reduces to the interval $[0,\rho (u,\bs v)]$,
where 
\begin{equation}\label{rhodef}
\rho \left( u,\bs v\right) :=\sup \left\{ r\geq 0:g\left( u,\bs m+rL\bs v\right)
\leq 0\right\}. 
\end{equation}
Note that $\rho \left( u,\bs v\right)=\infty$ is possible.
Hence, 
\[
\mu_\chi (\{r\geq 0:g(u,\bs m+rL\bs v\leq 0\})=\mu_\chi ([0,\rho (u,\bs v)])=F_\chi (\rho (u,\bs v)) - F_\chi (0)=F_\chi (\rho (u,\bs v)),
\]
where $F_\chi$ denotes the distribution function of the one-dimensional Chi-distribution with $n$ degrees of freedom. Here, to include the case $\rho \left( u,\bs v\right)=\infty$, we adopt the convention that $F_\chi (\infty)=1$. Consequently, the probability function simplifies to
\begin{equation}\label{spherint}
\varphi (u)=\int\limits_{\bs v\in\mathbb{S}^{n-1}}F_\chi (\rho (u,\bs v))d\mu_U(\bs v)\qquad (u\in U).
\end{equation}
For numerical purposes, the spherical integral \eqref{spherint} is usually approximated by a finite sum
\begin{equation}\label{srdestimator}
\varphi (u)\approx\tilde{\varphi}(u):= N^{-1}\sum\limits_{i=1}^NF_\chi \left(\rho (u,\bs v_i)\right)\qquad (u\in U),
\end{equation}
where $\{\bs v_i\}_{i=1}^N\subseteq\mathbb{S}^{n-1}$ are samples uniformly distributed on the sphere. Such samples can be obtained, for example, by normalizing to unit length a set of MC or quasi-MC samples of the standard Gaussian distribution $\mathcal N(0,I)$ in $\mathbb R^n$. For more efficient methods of creating such samples, see, e.g.,
\cite{Aistleitner2012}.

Although the following analysis focuses on Gaussian distributions, the methods also apply to other elliptical distributions, including multivariate t-distributions, symmetric Laplace, or hyperbolic distributions. The main difference lies in the precise form of the one-dimensional radial distribution (i.e., the Chi-distribution for the Gaussian case), whereas the spherical component remains consistent. Additionally, techniques from the Gaussian framework can be adapted to Gaussian-like distributions, including log-normal, truncated Gaussian, and Gaussian mixture distributions.

\subsection{Variance reduction by spherical-radial decomposition}\label{varreduc}
Here, we assume \eqref{slatercond} to ensure the validity of the form \eqref{spherint}. We will show that for large sets (and thus high probability), for small sets
(and thus low probability), and for radially near-symmetric sets about the mean, the variance ratio of the
probability estimators obtained with the SRD and with
standard MC approaches zero (as both tend to zero
individually). We fix a control $u$ (which does not play any role in this analysis) and recall the standard MC estimator %
\begin{equation}\label{eq:stdMC}
\varphi (u)\approx N^{-1}\#\{i\in\{1,\ldots ,N\}:g(u,\bs z_i)\leq 0\}=N^{-1}\sum\limits_{i=1}^N\mathcal{I}_{\{\bs z\in\mathbb{R}^n:g(u,\bs z)\leq 0\}}(\bs z_i),
\end{equation}
where $\{\bs z_i\}_{i=1}^N\subseteq\mathbb{R}^{n}$ are samples of
the Gaussian random vector $\bs \zeta$ and $\mathcal {I}_C$ refers to the
characteristic function of the set $C$. In contrast, the
estimator based on the SRD with the MC
samples $\{\bs v_i\}_{i=1}^N\subseteq\mathbb{S}^{n-1}$ on the sphere is \eqref{srdestimator}. Since the effect of the sample size $N$ on the variance is the same for both estimators, we may restrict ourselves to compare the variances of the elementary estimators 
\begin{equation}\label{SRD-MC}
\VMC:=\operatorname{Var}\mathcal {I}_{\{\bs z\in\mathbb{R}^n:g(u,\bs z)\leq 0\}}(\bs \zeta ),\qquad
\VSRD:=\operatorname{Var}F_\chi (\rho(u,\bs \theta )),
\end{equation}
where $\bs \zeta\sim\mathcal{N}(\bs m,\Sigma)$, and
$\bs \theta\sim\mathcal{U}(\mathbb{S}^{n-1})$ is uniformly distributed 
on the sphere.
It is well known that $\VMC=p(1-p)$, where 
\begin{equation}\label{pdef}
p=\varphi (u)=\mathbb{P}(\{ g(u,\bs \zeta)\leq 0\})=\mathbb{E}\mathcal {I}_{\{\bs z\in\mathbb{R}^n:g(u,\bs z)\leq 0\}}(\bs \zeta )=\mathbb{E}F_\chi (\rho(u,\bs \theta )).
\end{equation}
As shown in~\cite[eq.~(1.5)]{VanAckooijHenrion14}, the estimate $\VSRD\leq \VMC$ holds.
The following example shows
that the variances can be equal.
\begin{example}\label{varredcounterex}
Consider the function $g(u,\bs z):= z_1$, so that the inequality $z_1\leq
0$ defines a closed half-space with the origin on its boundary. Suppose
further that $\bs \zeta\sim\mathcal{N}(0,I)$. Then, clearly,
$p=1/2$, from which follows $\VMC=1/4$. On the other hand, the definition
\eqref{rhodef} of $\rho$ implies that
\[
\rho \left( u,\bs v\right) :=\sup \left\{ r\geq 0:r v_1\leq 0\right\}=\left\{\begin{array}{cc}
 0    & \mbox{if } v_1\geq 0,\\
\infty   & \mbox{if } v_1< 0.
\end{array}\right.
\]
Since $\bs \theta$ is uniformly distributed on the sphere, one concludes that, by symmetry, the events $\theta_1\geq 0$ and $\theta_1< 0$  occur with the same probability 1/2. Thus,
$
\mathbb{P}(\rho (u,\bs \theta)=0)=\mathbb{P}(\rho (u,\bs \theta)=\infty)=1/2
$,
and therefore both events $F_\chi (\rho (u,\bs \theta))=0$ and $F_\chi (\rho (u,\bs \theta))=1$ happen with probability 1/2. This means that $\mathbb{E}F_\chi (\rho (u,\bs \theta))=\mathbb{E}[F_\chi (\rho (u,\bs \theta))]^2=1/2$. Therefore,
\[
\VSRD=\mathbb{E}[F_\chi (\rho (u,\bs \theta))]^2-[\mathbb{E}F_\chi (\rho (u,\bs \theta))]^2 =1/4=p(1-p)=\VMC.
\]
\end{example}

This example shows that $\VSRD \le \VMC$ cannot be improved in general. However, in some situations, $\VSRD$ is substantially smaller than $\VMC$, as shown in the following two lemmas:
\begin{lemma}\label{lem:varred1}
With $p$ from \eqref{pdef} and $\rho^{\inf/\sup}:=\inf/\sup\{\rho (u,\bs v)\mid \bs v\in\mathbb{S}^{K-1}\}$, we have
\[
\VSRD\leq\min\{(1-p)(p-F_\chi(\rho^{\inf} )),p(F_\chi(\rho^{\sup})-p)\}.
\]
\end{lemma}
\begin{proof}
From \eqref{pdef}, it follows that
\[
\VSRD=\mathbb{E}[F_\chi (\rho (u,\bs \theta))]^2-[\mathbb{E}F_\chi (\rho (u,\bs \theta))]^2
\leq\mathbb{E}[F_\chi(\rho^{\sup})F_\chi (\rho (u,\bs \theta))]-p^2=p(F_\chi(\rho^{\sup})-p).
\]
For the second estimate, we use the identity $\operatorname{Var} X=\operatorname{Var} (1-X)$, which holds for any random variable $X$, and obtain that
\begin{align*}
\VSRD&= 
\mathbb{E}[1-F_\chi (\rho (u,\bs \theta))]^2-[\mathbb{E}(1-F_\chi (\rho (u,\bs \theta)))]^2\\&\leq\mathbb{E}[(1-F_\chi(\rho^{\inf}))(1-F_\chi (\rho (u,\bs \theta)))]-(1-p)^2\\
&=(1-F_\chi(\rho^{\inf}))(1-p)-(1-p)^2=(1-p)(p-F_\chi(\rho^{\inf} )).
\end{align*}
\end{proof}
This lemma implies that for $\rho^{\inf}\to\infty$ (growing sets with growing probability), one has that
\begin{equation}\label{eq:SRD/MC}
\frac{\VSRD}{\VMC}\leq\frac{(1-p)(p-F_\chi(\rho^{\inf}))}{p(1-p)}=1-\frac{F_\chi(\rho^{\inf})}{p}\to 0.
\end{equation}
Here, we used that $\rho^{\inf}\to\infty$ implies $F_\chi(\rho^{\inf})\to 1$ and that $F_\chi(\rho^{\inf})\leq\mathbb{E}F_\chi (\rho (u,\bs \theta))=p$ necessarily implies $p\to 1$. Thus, the ratio between the variances of the SRD and the standard MC estimators becomes arbitrarily small for high probabilities $p$. 

Similarly, for $\rho^{\sup}\to 0$ (shrinking sets with decreasing probability, yet always containing the mean $\bs m$ of the given Gaussian distribution), one has that
\[
\frac{\VSRD}{\VMC}\leq\frac{p(F_\chi(\rho^{\sup})-p)}{p(1-p)}=\frac{F_\chi(\rho^{\sup})-p}{1-p}\to 0.
\]
Similar as above, we used that $F_\chi(\rho^{\sup})\to 0$ for $\rho^{\sup}\to 0$, and that $F_\chi(\rho^{\sup})\geq p$ also implies $p\to 0$. This means that the ratio between the variances of the SRD and the standard MC estimators goes to zero as $p\to 0$ (under the additional assumption that the mean of the Gaussian distribution is contained in the set, which---following the remarks below \eqref{slatercond}---was automatically satisfied in the first case because $p\geq 0.5$ may be assumed due to $p\to 1$).

We may interpret $\rho_{\delta} := \rho^{\sup}-\rho^{\inf}\geq 0$ as a measure of radial asymmetry of the set $\{\bs z\in\mathbb{R}^n: g(u,\bs z)\leq 0\}$. Clearly, for $\rho_{\delta} =0$, one has that $\rho (u,\bs v)$ is constant for all $\bs v\in\mathbb{S}^{n-1}$ so that the set becomes perfectly symmetric (i.e., a ball) around zero. 
\begin{lemma}\label{lem:varred2}
Under the assumptions of \cref{lem:varred1} and also assuming that $p\in (0,1)$, one has that $\VSRD\leq L^2 \rho_{\delta}^2$, where $L$ is the maximum of the density  $\mu_\chi =F'_\chi$ of the given $\chi$-distribution (which is bounded for any dimension).
\end{lemma}
\begin{proof}
Thanks to $p\in (0,1)$, there exists a unique $p$-quantile $\rho^*$ of the one-dimensional Chi-distribution such that $F_\chi (\rho^*)=p$. One easily derives $\rho^*\in [\rho^{\inf},\rho^{\sup}]$ because otherwise a contradiction with $\mathbb{E}F_\chi (\rho (u,\bs \theta))=p$ would arise from the fact that the distribution function $F_\chi$ is strictly increasing. Now, exploiting this last property once more, we arrive at
\begin{align*}
\VSRD&=\operatorname{Var}(F_\chi (\rho(u,\bs \theta ))=\int\limits_{\bs v\in\mathbb{S}^{K-1}}[F_\chi (\rho(u,\bs v ))-\mathbb{E}F_\chi (\rho(u,\bs v ))]^2d\mu_U (\bs v)\\
&=\int\limits_{\bs v\in\mathbb{S}^{K-1}}[F_\chi (\rho(u,\bs v ))-F_\chi (\rho^*)]^2d\mu_U (\bs v)\leq (F_\chi (\rho^{\sup})-F_\chi (\rho^{\inf}))^2\\
&=(F'_\chi (\tilde{\rho})(\rho^{\sup}-\rho^{\inf}))^2\leq L^2\rho_{\delta}^2\qquad (\tilde{\rho}\in [\rho^{\inf},\rho^{\sup}]).
\end{align*}
\end{proof}

Note that \Cref{lem:varred2} implies that the SRD estimator has zero variance for balls centered at zero. Therefore, a single ray is sufficient to compute the exact probability of a ball.
This is of course well-known.
Now, consider a sequence of sets with the same probability $p$ converging to a centered ball. In this sequence, $\VMC=p(1-p)$ is the fixed variance of the MC estimator, while $\VSRD$ converges to zero. Thus, the variance ratio between the estimators reduces in favor of the SRD as shown in \Cref{lem:varred1}.

Finally, we revisit \Cref{varredcounterex}. For the half-space
considered there, $\rho^{\inf}=0$ and $\rho^{\sup}=\infty$. Hence, the estimate from \cref{lem:varred1} only yields the already known relation $\VSRD\leq p(1-p)=\VMC$. Moreover, $\rho_{\delta}=\infty$, whence \Cref{lem:varred2} does not provide additional information either. The reason is that this half space is neither a large nor a small space (its probability is equal to $1/2$, which is far from $1$ and $0$) nor a symmetric set. This explains why in this situation the variance of the SRD MC estimator is not reduced compared to the classical MC estimator.

\subsection{Differentiability and computation of gradients of the probability function}\label{sec:spherical-radial}
When numerically addressing a chance constraint such as \eqref{eq:chance}, compactly written as $\varphi (u)\geq p$ with $\varphi$ defined in \eqref{probfunc}, it is essential to compute not only the value of $\varphi(u)$, but also its derivatives. Analytical properties of $\varphi$, such as Lipschitz continuity or differentiability, have been studied before \cite{VanAckooijHenrion14,VanAckooijHenrion17,HantouteHenrionPerez19,vanAckooij2024}.
Given our focus on numerical methods, we concentrate on the discretized probability function $\tilde{\varphi}$ defined in \eqref{srdestimator}, which uses the samples $\{\bs v_i\}_{i=1}^N\subseteq\mathbb{S}^{n-1}$. Specifically, we aim at identifying conditions guaranteeing its differentiability and derive an explicit form of the derivative suitable for implementation. Since the Chi-distribution function $F_\chi$ is continuously differentiable, with its derivative being the density $\mu_\chi$, the problem reduces to verifying the partial differentiability of the function $\rho(u,\cdot)$ with respect to $u$. We now fix some $\bar{u}\in X$ and assume that the random constraint function $g$ from \cref{genset} has the typical form relevant for joint chance constraints, namely
\begin{equation}\label{gmaxdef}
g:=\max\limits_{j=1,\ldots ,p}g_j,
\end{equation}
where the $g_j:U\times\mathbb{R}^n\to\mathbb{R}$ are affine. In particular, $g(u,\cdot)$ is convex for all $u\in U$ as required in \cref{genset}. Analogously to \eqref{rhodef}, we define %
\begin{equation}\label{rhoidef}
\rho_j \left( u,\bs v\right) :=\sup \left\{ r\geq 0:g_j\left(u,\bs m+rL\bs v\right)
\leq 0\right\}\qquad (j=1,\ldots ,p).     
\end{equation}
Under the additional assumption
\begin{equation}\label{slateri}
 g_j(\bar{u},\bs m)<0\qquad (j=1,\ldots ,p)  
\end{equation}
it follows that \eqref{slatercond} holds. Moreover,
\begin{equation}\label{rhomax}
\rho \left(\bar{u},\bs v\right)=\min\limits_{j=1,\ldots ,p}\rho_j \left( \bar{u},\bs v\right).   
\end{equation}
In the following we denote by $f_\chi$ the density of the one-dimensional Chi-distribution $\mu_\chi$ with $n$ degrees of freedom and we adopt the convention that 
$f_\chi (\infty)=\lim\limits_{t\to\infty}f_\chi (t)=0$.
\begin{proposition}\label{derivformula}
Let $U$ be a Banach space. In the setting above and under assumption \eqref{slateri}, if for all $i\in\{1,\ldots ,N\}$ satisfying $\rho (\bar{u},\bs v_i)<\infty$ holds
\begin{equation}\label{cq}
\#\{j\in\{1,\ldots , p\}:\rho \left(\bar{u},\bs v_i\right)=\rho_j \left( \bar{u},\bs v_i\right) \}=1,
\end{equation}
then $\tilde{\varphi}$ is continuously differentiable at $\bar{u}$ with derivative
\begin{equation}\label{gradform}
\nabla\tilde{\varphi}(\bar{u})=-N^{-1}\sum\limits_{i=1}^N\frac{f_\chi (\rho_{j_i} (\bar{u},\bs v_i))}{\langle \nabla_zg_{j_i}(\bar{u},\bs m+\rho_{j_i}(\bar{u},\bs v_i)L\bs v_i),L\bs v_i\rangle}\nabla_ug_{j_i}(\bar{u},\bs m+\rho_{j_i}(\bar{u},\bs v_i)L\bs v_i).
\end{equation}
Here, for $i\in\{1,\ldots ,N\}$, the index $j_i$ is one satisfying \eqref{cq}, which is unique if $\rho (\bar{u},\bs v_i)<\infty$. 
\end{proposition}
\begin{proof}
Our setting allows us to invoke \cite[Cor.\ 3.2]{vanAckooij2024} for each $j\in\{1,\ldots ,p\}$ to argue that the functions $F_\chi (\rho_j (u,\bs v))$ are continuously differentiable at all $(u,\bs v)\in\mathcal{V}_j(\bar{u})\times\mathbb{S}^{n-1}$, where $\mathcal{V}_j(\bar{u})$ is some neighborhood of $\bar{u}$. Moreover, for all $\bs v\in\mathbb{S}^{n-1}$, one has that 
\begin{equation}\label{single}
\nabla [F_\chi (\rho_j (\cdot,\bs v))](\bar{u})=\frac{-f_\chi (\rho_{j} (\bar{u},\bs v))}{\langle \nabla_z g_j(\bar{u},\bs m+\rho_{j}(\bar{u},\bs v)L\bs v),L\bs v\rangle}\nabla_ug_j(\bar{u},\bs m+\rho_{j}(\bar{u},\bs v)L\bs v).
\end{equation}
Actually, the application of \cite[Cor.~3.2]{vanAckooij2024} requires the verification of a certain growth condition. However, since $g_j$ is affine, this condition is easily verified along the lines of \cite[Lem.~6]{geiersbach2023optimality}. 
To prove \eqref{cq}, consider first an index $i\in\{1,\ldots ,N\}$ with $\rho (\bar{u},\bs v_i)<\infty$. It follows from \cite[Lem.~1]{HantouteHenrionPerez19} and \eqref{cq} that $\rho (\cdot\bs ,\bs v_i)=\rho_{j_i} (\cdot,\bs v_i)$ locally around $\bar{u}$.
Therefore, by \eqref{single}
\begin{equation}\label{single2}
\nabla [F_\chi (\rho (\cdot,\bs v_i))](\bar{u})=\frac{-f_\chi (\rho_{j_i} (\bar{u},\bs v_i))}{\langle \nabla_z g_{j_i}(\bar{u},\bs m+\rho_{j_i}(\bar{u},\bs v_i)L\bs v_i),L\bs v_i\rangle}\nabla_ug_{j_i}(\bar{u},\bs m+\rho_{j_i}(\bar{u},\bs v_i)L\bs v_i),
\end{equation}
which yields \eqref{gradform}. Now, consider an index $i$ where $\rho (\bar{u},\bs v_i)=\infty$, whence $\rho_j (\bar{u},\bs v_i)=\infty$ for all $j=1,\ldots, p$. Consequently, $F_\chi (\rho (\bar{u},\bs v_i))=F_\chi (\rho_j (\bar{u},\bs v_i))$ for all $j\in\{1,\ldots , p\}$.
By the remark above, all $F_\chi (\rho_j (\cdot ,\bs v_i))$ are continuously differentiable at $\bar{u}$ with vanishing derivative \cite[Lem.~3.3]{vanAckooij2024}. This shows that $F_\chi (\rho (\cdot,\bs v_i))$ itself is continuously differentiable with zero derivative at $\bar{u}$ and, hence, \eqref{single2} also holds in this second case thanks to $f_\chi (\infty)=0$. Finally, \eqref{gradform} follows from \eqref{single2} and the definition of $\tilde{\varphi}$ in \eqref{srdestimator}.
\end{proof}
The next proposition shows that the difficult-to-verify condition \eqref{cq}, required for differentiability of $\tilde\varphi$, is generically satisfied under a common constraint qualification that
is weaker than the prominent \textit{Linear Independence Constraint Qualification} (LICQ):
\begin{proposition}\label{prop:r2cq}
If the following so-called 'Rank-2-Constraint Qualification'
\begin{eqnarray}
&\forall \bs z\in\mathbb{R}^n, i,j\in\{1,\ldots ,p\}, i\neq j:&\notag\\
&0=g(\bar{u},\bs z)=g_i(\bar{u},\bs z)=g_j(\bar{u},\bs z)\Longrightarrow\textrm{ rank }\{\nabla_{\bs z} g_i(\bar{u},\bs z),\nabla_{\bs z}g_j(\bar{u},\bs z)\}=2 &\label{r2cq}
\end{eqnarray}
holds at $\bar{u}\in U$, then a random sample 
$\{\bs v_i\}_{i=1}^N\subseteq\mathbb{S}^{n-1}$ satisfies \eqref{cq} with probability one.
\end{proposition}
\begin{proof}
Following \cite[Lem.~4.3]{VanAckooijHenrion17}, \eqref{r2cq} implies that
\[
\mu_U(\{\bs v_i\in\mathbb{S}^{n-1}:\rho (\bar{u},\bs v)<\infty,\,\#\{j\in\{1,\ldots , p\}:\rho \left(\bar{u},\bs v_i\right)=\rho_j \left( \bar{u},\bs v_i\right) \}=1\})=1.
\]
This shows that \eqref{cq} is satisfied with probability one for each $i=1,\ldots ,N$ separately, hence it is satisfied with probability one simultaneously for all $i=1,\ldots , N$. 
\end{proof}

\section{Application of spherical-radial decomposition to the control problem}\label{sec:application}
\subsection{Discretization of joint state constraints}
To replace the continuous state constraint with a finite number of
inequalities, we use a continuous linear map $E:Y\to \mathbb R^M$. If
$Y$ continuously embeds into the space of continuous function, then
the map $E$ can be chosen as the evaluation of the state
$y(\omega)$ at the points $\{ \bs x_{1},\ldots ,\bs x_M\} \subseteq
\mathcal{D}$. If the dimension of $\D$ is $d=1$, $Y\subset
H^1(\D)\hookrightarrow \mathcal C^0(\D)$ and thus point evaluation is
continuous.  If $d=2,3$ and the data is smooth, then additional
regularity arguments for elliptic operators can be used to
show that $Y\subset H^2(\D)$, which again continuously embeds into the
space of continuous functions, and thus choosing $E$ as point
evaluation is feasible. If $d>3$ or no additional regularity can be
used, $E$ can be chosen as local averages of the state variable
$y(\omega)$ centered at the points $\{ \bs x_{1},\ldots ,\bs x_M\}$.  For ease
of notation, in the following we write $E$ as point evaluation, i.e.,
$E(y):= (y(\bs x_1),\ldots,y(\bs x_M))$. The following derivations also apply to $E$ being a local averaging operator.  Note that here we
consider the points $\bs x_j$ as given, e.g., as a uniform grid of
points covering the domain $\D$. It is also possible to choose the set of points $\{ \bs x_{1},\ldots ,\bs x_M\}$ adaptively using probabilistic information, thus keeping the number of points $M$ moderate, 
\cite{BertholdHeitschHenrionEtAl21}. Now, with no matter which choice of a discrete subset $\{ \bs x_{1},\ldots ,\bs x_M\}$ of the domain, we end up at the accordingly modified chance constraint \eqref{eq:chance}:
\begin{equation}\label{eq:chancedisc}
  \mathbb P(\omega\in \Omega \,| \, \underline y( \bs x_j) \le y( \bs x_j, \omega)\le \bar y(\bs x_j)
  \quad\forall j=1,\ldots ,M) \ge p,
\end{equation}
To write the probability in this chance constraint 
for the concrete governing equation \eqref{eq:redstate} in the form $\mathbb{P}(\max_{1\le j\le M} g_j(u,\bs z ))$,
we define the functions $g_j$ %
as 
\begin{equation}\label{gidef}
\begin{array}{rcl}
g_j(u,\bs z)&:=&\left[A(u)^{-1}\left(f- Cu-B\xi_0 - \sum_{k=1}^Kz_k
    Be_{k}\right) - \bar{y}\right](\bs x_j)\\[1.3ex]
g_{M+j}(u,\bs z)&:=&\left[\underline{y}-A(u)^{-1}\left(f- Cu-B\xi_0 - \sum_{k=1}^Kz_kBe_{k}\right) \right](\bs x_{j}),
\end{array}
\end{equation}
where we assume the form \eqref{afflinrand} for the random process $\xi$, with the concrete distribution $\bs \zeta\sim\mathcal{N}(0,\Sigma)$. In particular, $\bs z\in\mathbb{R}^K$ represents the realizations $\bs \zeta (\omega)$ of $\bs \zeta$. This fits our abstract setting from \cref{sec:spherraddec} with the choices $p:=2M, n:=K$. Clearly, the $g_j$ are affine in $\bs z$ for $j=1,\ldots,2M$ as required in \cref{sec:spherical-radial}.
\subsection{Implicit dimension reduction in probability space}\label{subsec:KL}
We now turn to the approximation in random space. A standard method to efficiently describe a Gaussian random variable or field is its Karhunen-Loeve (KL) expansion. For instance, for a random field $\xi\sim\mathcal N(\xi_0,\mathcal C_0)$ over the physical domain $\D$ such a (possibly infinite) expansion is 
\begin{equation}\label{eq:kl-xi}
    \xi \left(\omega \right) = \xi_0 + \sum_k %
    \zeta_k^\xi(\omega) e_{k}^\xi,
\end{equation}
where $e_k^\xi\in L^2(\D)$ are eigenvectors (or functions) of $\mathcal C_0$, and $\zeta_k^\xi(\omega)$ are independent Gaussian random variables with variances $\lambda_k^\xi$, where $\lambda_1^\xi\ge \lambda_2^\xi\ge \ldots \ge 0$ are the eigenvalues corresponding to $e_k^\xi$.  The rate at which the sequence $\{\lambda_k^\xi\}_{k}$ decays controls the approximation error when the KL expansion \eqref{eq:kl-xi} is truncated. 

Although a truncated version of \eqref{eq:kl-xi} is an efficient (in terms of the number of terms and thus the dimension) representation for the random variable $\xi(\omega)$, it may not be efficient for the state random variable $y(\omega)$, which is subject to the joint chance constraints. To find a representation tailored to the state, note that \eqref{eq:controltostate} implies that
\begin{equation}\label{eq:kl-eta}
y(\omega) =  %
y_{0}^u - A(u)^{-1}B\xi(\omega) =:y_{0}^u - \sum_{k} 
\zeta_{k}^y(\omega)  e_{k}^y, %
\end{equation}
where $y_{0}^u := A(u)^{-1}(f-B\xi_0-Cu)$ is the mean of the state variable. For fixed $u$, $e_k^y$ are the eigenfunctions of 
$A(u)^{-1}B\mathcal C_0 B^\star A(u)^{-\star}$, with the ``$-\star$'' denotes the inverse adjoint operator. Moreover, in this KL-type expansion, $\zeta_k^y(\omega)$ are independent Gaussian random variables with variance $\lambda_k^y$, where $\lambda_1^y\ge \lambda_2^y\ge \ldots \ge 0$ are the eigenvalues corresponding to $e_k^y$. 

Due to smoothing properties of the inverse elliptic operator
$A(u)^{-1}$, the variances $\lambda_k^y$ of the random variables in \eqref{eq:kl-eta} typically decay faster than the variances $\lambda_k^\xi$ for \eqref{eq:kl-xi}. Thus, using \eqref{eq:kl-eta} instead of \eqref{eq:kl-xi} allows a more efficient finite-dimensional approximation of the random variable. To illustrate this, we show in the numerical example in \cref{sec:example-linear} that $K=20$ coefficients allow a highly accurate representation of $y(\omega)$, while an accurate representation of $\xi(\omega)$ would require a substantially larger number of expansion coefficients; see \cref{fig:C0}.

Note that in \eqref{eq:kl-xi}, $\lambda_k^\xi$ and $e_k^\xi$ depend only on the distribution of the random variable but not on the control $u$. When the PDE operator $A$ is independent of $u$, only the mean $y_{0}^u$ in \eqref{eq:kl-eta} depends on the control, whereas $\lambda_k^y$ and $e_k^y$ are independent of $u$. For both scenarios, eigenfunctions and eigenvalues can be computed upfront and used throughout a numerical optimization algorithm to find the optimal control.
However, if $A(u)$ is a function of $u$, then $\lambda_k^y$ and $e_k^y$ are also functions of $u$ and need to be recomputed whenever the control changes (such as in a numerical optimization algorithm). Thus, while \eqref{eq:kl-eta} provides an efficient way to represent the random state variable, in situations where the PDE operator depends on $u$, utilizing \eqref{eq:kl-eta} may not be computationally efficient to evaluate the probability $\varphi(u)$ (or its approximation \eqref{srdestimator}), and its derivative with respect to $u$.

\subsection{Probability function and its gradient for the concrete control problem}
We start by observing that the condition \eqref{slateri} at some $u\in U$ translates in our concrete control problem via \eqref{gidef} and by virtue of $\bs m=0$ as
\begin{equation}\label{slatericonc}
\underline{y}(\bs x_j)<y_{0}^u(\bs x_j)<\bar{y}(\bs x_j)\quad (j=1,\ldots ,M),
\end{equation}
where $y_{0}^u:=A(u)^{-1}(f-Cu-B\xi_0)$ refers to the state associated with the control $u$ and the mean field $\xi_0$ of the random source term $\xi$ according to \eqref{afflinrand}. This being satisfied, we may represent the sample-based approximation of the probability function associated with our concrete control problem as \eqref{srdestimator} with $\rho(\cdot,\cdot)$ defined in \eqref{rhodef}. In order to make formula \eqref{srdestimator} explicit, we have to represent $\rho$  in terms of the original data of our control problem.
Using \eqref{gidef}, the functions $\rho_j$ from \eqref{rhoidef} calculate, for $j=1,\ldots ,M$ and $\bs v\in\mathbb{S}^{K-1}$, as
\begin{eqnarray}\label{eq:rho_j}
\rho_j \left( u,\bs v\right)&=&\left\{
\begin{array}{cc}
\frac{y_{0}^u(\bs x_j)-\bar{y}(\bs x_j)}{\delta (u,\bs v)(\bs x_j)}     & \mbox{if }\delta (u,\bs v)(\bs x_j)< 0 \\
\infty &\mbox{if }\delta (u,\bs v)(\bs x_j)\geq 0
\end{array}
\right.\\
\label{eq:rho_Mj}
\rho_{M+j} \left( u,\bs v\right)&=&\left\{
\begin{array}{cc}
\frac{y_{0}^u(\bs x_j)-\underline{y}(\bs x_j)}{\delta (u,\bs v)(\bs x_j)}     & \mbox{if }\delta (u,\bs v)(\bs x_j)> 0 \\
\infty &\mbox{if }\delta (u,\bs v)(\bs x_j)\leq 0
\end{array}
\right.
\end{eqnarray}
Here, $\delta (u,\bs v):=A(u)^{-1}B\sum_{k=1}^K(L\bs v)_ke_k$. Now, the approximating probability function $\tilde{\varphi}(u)$ in \eqref{srdestimator} is immediately calculated from the above formulas using the representation \eqref{rhomax}.
Regarding the derivative of $\tilde{\varphi}$, we can calculate the partial derivatives of $g_j$ with respect to $u$ and $\bs z$ as in
\eqref{gidef}. Doing so, we obtain for all $h\in U$, $j=1,\ldots, p$, and $k=1,\ldots ,K$:
\begin{eqnarray}
\langle\nabla_ug_j(u,\bs z),h \rangle &=&-\langle\nabla_ug_{M+j}(u,\bs z),h \rangle \nonumber \\
&=&\left[\left(\left(A(u)^{-1}\right)'h\right)\left(f- Cu-B\xi_0 - \sum_{k=1}^Kz_kBe_{k}\right)-A(u)^{-1}Ch\right](\bs x_j) \label{eq:nabla_u}\\
\frac{\partial g_j}{\partial z_k}(u,\bs z)&=&-\frac{\partial g_{M+j}}{\partial z_k}(u,\bs z)=-\left[A(u)^{-1}Be_k\right](\bs x_j) \label{eq:nabla_z}
\end{eqnarray}
Now, under assumption \eqref{cq}, we obtain that $\tilde{\varphi}$ is continuously differentiable at $\bar{u}$ with the derivative in \eqref{gradform}, which can be made explicit using \eqref{eq:nabla_u},\eqref{eq:nabla_z}. 

Finally, we make the assumption \eqref{cq} more explicit in light of the rank-2 constraint qualification in \cref{prop:r2cq}. For that purpose, for a fixed control $\bar{u}$, we denote by $y^{(k)}$ the solutions of
\[
A(\bar{u})y+Be_k=0\quad (k=1,\ldots ,K).
\]
\begin{proposition}
If for any $i,j\in\{1,\ldots ,M\}$ with $i\neq j$, the two vectors
\[
\left(\begin{array}{c}y^{(1)}(\bs x_i)\\ \vdots\\y^{(K)}(\bs x_i)\end{array}\right),\left(\begin{array}{c}y^{(1)}(\bs x_j)\\ \vdots\\y^{(K)}(\bs x_j)\end{array}\right)
\]
are linearly independent, then the Rank-2-Constraint Qualification from \cref{prop:r2cq} is satisfied at $\bar{u}$ and, consequently, for a random sample $\{\bs v_i\}_{i=1}^N\subseteq\mathbb{S}^{n-1}$, the probability function $\tilde{\varphi}$ from \eqref{srdestimator} is continuously differentiable at $\bar{u}$ with probability one.
\end{proposition}
\begin{proof}
First, note that due to \eqref{slatericonc}, the constraint functions \eqref{gidef} cannot satisfy the condition
$0=g_i(\bar{u},\bs z)=g_{M+i}(\bar{u},\bs z)$
for some $\bs z\in\mathbb{R}^n$ and some $i\in\{1,\ldots ,M\}$. Therefore, the set $\{i,j\}\subseteq\{1,\ldots ,p=2M\}$ with $i\neq j$ in \eqref{r2cq} can always be written in one of the forms
\[
\{i,j\}=\{i',j'\}\quad\mbox{or}\quad \{i,j\}=\{i',M+j'\}\quad\mbox{or}\quad \{i,j\}=\{M+i',M+j'\}
\]
for some $i',j'\in\{1,\ldots M\}$ with $i'\neq j'$. In any case, due to $\eqref{eq:nabla_z}$ and the definition of $y^{(k)}$, the gradient $\nabla_{\bs z} g_i(\bar{u},\bs z)$ coincides, up to the sign, with the first vector above. Likewise, $\nabla_{\bs z} g_j(\bar{u},\bs z)$ coincides (up to the sign) with the second vector above. Therefore, our assumption on the linear independence of these two vectors implies  
the conclusion of 
\eqref{r2cq} .
\end{proof}

\section{Numerical results for linear governing equation}
\label{sec:example-linear}
We now use the linear-quadratic problem from \cref{ex:linear} to study
the approximation of the joint chance constraint and the behavior of
optimal controls. First, we
show that the integration over $\Omega$ in the objective
\eqref{eq:ex1-J} can be done analytically for a Gaussian random
field. For that purpose, assume that the law of the Gaussian
measure is $\mathcal N(\xi_0,\mathcal C_0)$.  Using the linearity of
the governing equation implies that, for fixed control $u$, the state
$y$ is also Gaussian. Thus, integration over $\Omega$ of the
quadratic in \eqref{eq:ex1-J} can be performed analytically,
\cite[Remark 1.2.9]{DaPratoZabczyk02}, and results in the reduced
problem
\begin{equation}\label{eq:optcon-red}
        \begin{aligned}
        & \underset{\mathclap{u\in L^2(\D)}}{\text{minimize}}
        & & \tilde{\mathcal J}(u) := \frac 12
          \int_\D (y_{0}^u - y_d)^2 \,d \bs x
  + \frac\alpha 2 \int_\D u^2\,d\bs x + \text{Tr}(A^{-1}B\mathcal C_0
  B^\star A^{-\star}),\\
  &\text{subject to } && \text{the joint chance constraint }\eqref{eq:chance},
        \end{aligned}
\end{equation}
where $A=-\Delta$ and $B:L^2(\partial \D_2)\mapsto H^{-1}(\D)$ and
$y_{0}^u = A^{-1}(f+u + B\xi_0)$ is the mean of the distribution of
states. Moreover, ``Tr'' denotes the trace of an operator. Note that the
last term in~\eqref{eq:optcon-red} is independent of $u$ and thus can
be neglected in optimization. Next, we describe the concrete domain and parameters for the problem \eqref{eq:ex1-J},
\eqref{eq:state_linear}, and \eqref{eq:chance}.

We choose the domain $\D=(0,1)^2\subset \mathbb R^2$ and divide the boundary into $\partial D_2=\{0\}\times [0,1]$ and $\partial
D_1=\partial D\setminus \partial D_2$. The Laplace operator uses zero
Dirichlet boundary conditions on $\partial D_1$ and Neumann boundary
conditions on $\partial D_2$. In addition, $y_d = \frac{1}{10} \cos(2\pi
x_1)\sin(2\pi x_2)$, $f\equiv 0$, $\alpha=10^{-5}$. The uncertain parameter field enters as Neumann data on the one-dimensional domain
$\partial D_2$. This data follows an infinite-dimensional Gaussian
distribution with mean $\xi_0 \equiv 0$, and a covariance operator 
given by the inverse elliptic PDE operator $\mathcal C_0 =
\gamma(-\partial_{{x_2 x_2}})^{-1}$, with homogeneous Dirichlet conditions at the boundary of $\partial D_2$, i.e., at the
two points $(0,0)$ and $(0,1)$, and with $\gamma=4$. This covariance
operator has the eigenfunctions $\sin(k\bs x_2)$, $k=1,2,\ldots$ with
corresponding eigenvalues $4(k\pi)^{-2}$. Samples from this
distribution, and of the distribution of the two-dimensional state variable $y(\omega)$ are shown in
\cref{fig:C0}. This figure also shows a comparison of the KL expansion eigenvalues of $\xi$ and $y$, see \eqref{eq:kl-xi} and \eqref{eq:kl-eta}, respectively. As can be seen, the KL factors $\lambda_i^y$ of $y(\omega)$ decay much faster despite corresponding to a Gaussian random field defined over a two-dimensional domain (note that $\xi$ is defined over a one-dimensional domain). This faster decay shows that the random variable $y(\omega)$ can be well approximated in a lower dimension $K$.

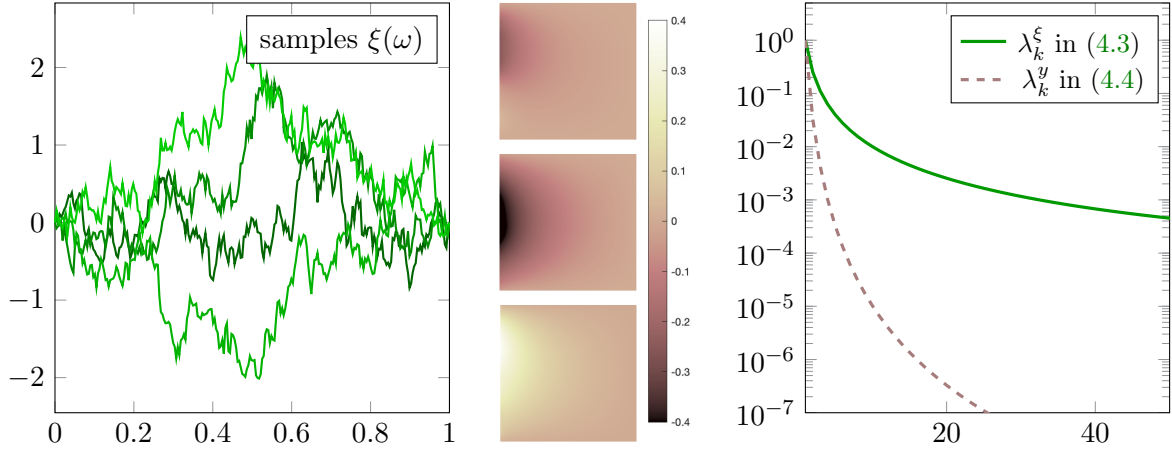
\begin{figure}[bt]\centering
  \begin{tikzpicture}
    \begin{axis}[height=7cm, width=6.8cm,
        xmin=0,
        xmax=1,
        compat=1.3,
        legend pos= north east,
        legend style={empty legend}]
    \addplot[color=green!40!black,mark=none,thick] table
        [x index=0,y index=1]{data/samples-of-xi.dat};
    \addplot[color=green!55!black,mark=none,thick] table
        [x index=0,y index=2]{data/samples-of-xi.dat};
    \addplot[color=green!70!black,mark=none,thick] table
        [x index=0,y index=3]{data/samples-of-xi.dat};
    \addplot[color=green!80!black,mark=none,thick] table
        [x index=0,y index=4]{data/samples-of-xi.dat};
      \addlegendentry{samples $\xi(\omega)$}
    \end{axis}
  \end{tikzpicture}\hspace{.2cm}
  \begin{tikzpicture}
  \node at (0,5cm) {\includegraphics[width=1.8cm]{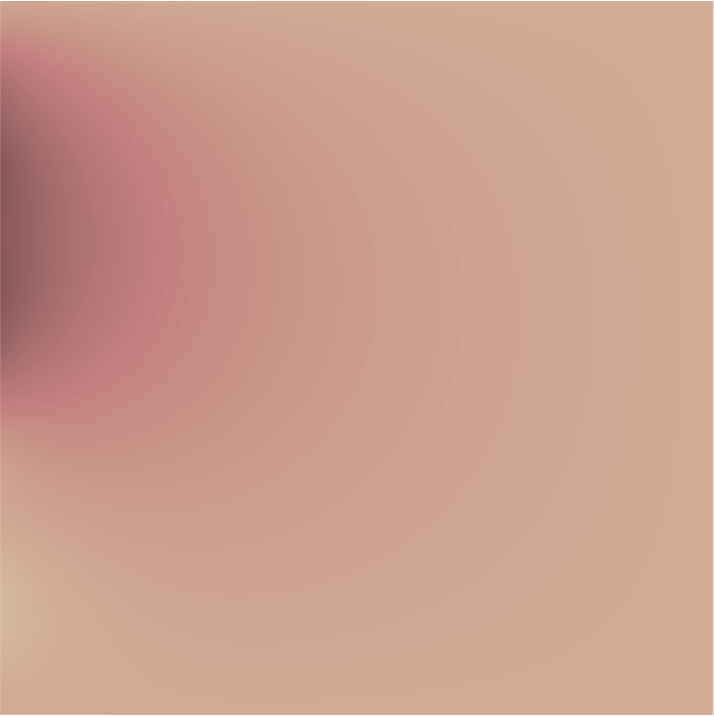}};
   \node at (0,3cm) {\includegraphics[width=1.8cm]{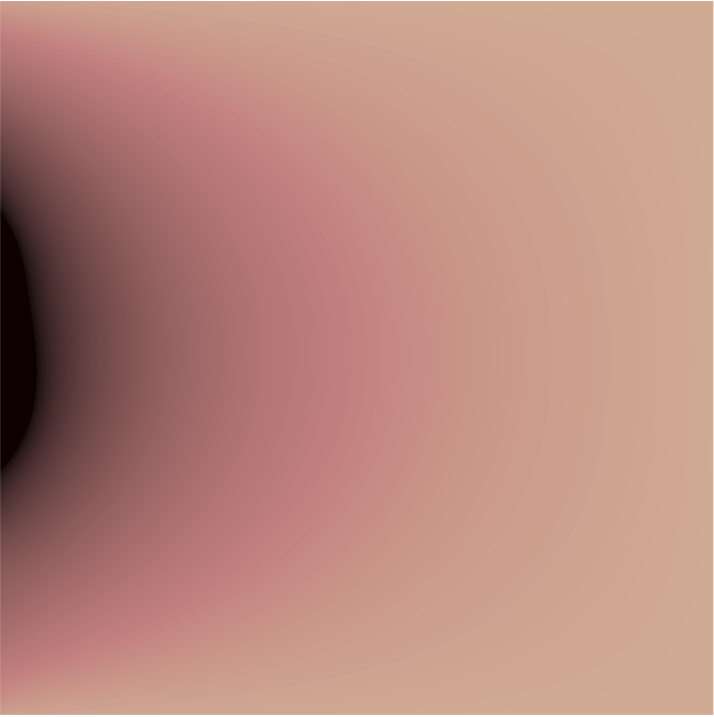}};
    \node at (0,1cm) {\includegraphics[width=1.8cm]{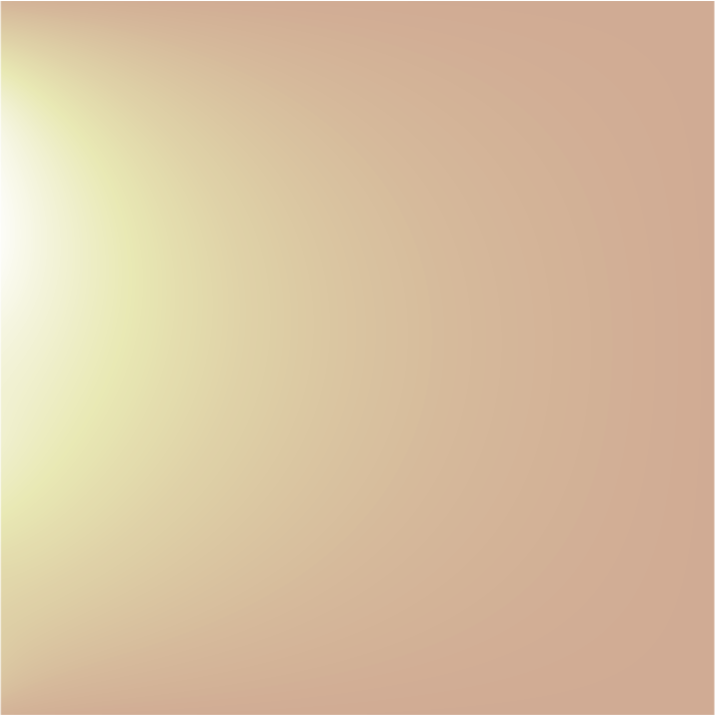}};
   \node at (1.3,3cm) {\includegraphics[width=0.7cm]{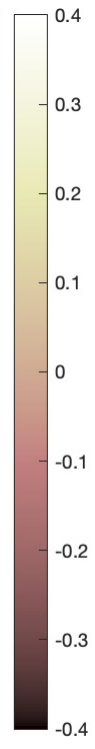}};
  \end{tikzpicture}\hspace{.2cm}
  \begin{tikzpicture}
    \begin{semilogyaxis}[height=7cm, width=6.4cm,
        xmin=1,
        xmax=50,
        ymin=1e-7,
        compat=1.3,
        legend pos= north east]
      \addplot[color=green!60!black,mark=none,very thick] table
        [x index=0,y index=1]{data/KLcoeffs128.dat};
          \addlegendentry{\small $\lambda_k^\xi$ in \eqref{eq:kl-xi}}
        \addplot[color=pink!65!black,mark=none,very thick,dashed] table
                [x index=0,y index=2]{data/KLcoeffs128.dat};
        \addlegendentry{\small \ $\lambda_k^y$ in \eqref{eq:kl-eta}}
    \end{semilogyaxis}
  \end{tikzpicture}
  \caption{The left figure shows samples from the distribution of the random boundary data
    $\xi(\omega)$. The middle figure shows samples from the distribution of $y(\omega)-y_{0}^u$, the state with zero mean. 
    The right figure shows the normalized (i.e., $\lambda_1$ is scaled to $1$) eigenvalue factors in the KL expansion \eqref{eq:kl-xi} of $\xi$ (green, solid), and in the KL expansion \eqref{eq:kl-xi} of $y$ (dark pink, dashed). 
 \label{fig:C0}}
\end{figure}

\subsection{Computational aspects}
A finite difference approximation (i.e., the five-point stencil) is used on a mesh of $n\times n$ points to discretize the Laplacian $\Delta$ in the governing equation on the spatial domain $\D$. Finite differences are also used to discretize the one-dimensional second derivative operator in the definition of the covariance $\mathcal C_0$. Solving \eqref{eq:optcon-red} requires repeated solution of the governing equation and thus inversion of the discretization of $A$. Thus, we compute the Choleski factorization of the system matrix upfront and reuse these factors throughout the solution process. 
Unless otherwise specified, we use $n=128$, i.e., the discretized state has a dimension of $n^2=16,384$. We discretize the state constraint on the same grid as the governing equation, i.e., $M=n^2$. Due to the rapid decay of the Karhunen-Loeve factors for the expansion of $y(\omega)$ shown in \Cref{fig:C0}, we use $K=20$ in the expansion, unless otherwise specified.
To obtain samples of the uniform distribution on the sphere, we first generate MC or QMC samples from the standard Gaussian distribution in $\mathbb R^K$, where we use the Halton sequence for QMC. These samples are then normalized to unit length, a well-known technique for obtaining samples from the uniform distribution on the sphere. In the context of QMC, this approach can be further improved by using more efficient low-discrepancy sequences specifically designed for the sphere, e.g., \cite{Aistleitner2012}.

\subsection{Approximation of chance constraint probability function}
First, we compare different approximations of the
probability \eqref{eq:chance}, where all tests are performed at the
nominal control $u=\frac 15 \sin(2\pi x_1)\cos(\pi x_2)$. We compare different sampling methods for an increasing number of samples, namely standard MC sampling in $\mathbb{R}^K$, as well as spherical-radial MC and quasi Monte Carlo (QMC) sampling on $\mathbb{S}^{K-1}$ . 

The results for the root mean square error for two different sets of upper and lower bounds for the state are shown in \cref{fig:chance-conv}. Here, we used $10^8$ samples from the standard MC method to compute a highly accurate estimate of the exact probability. 
We observe the expected
$1/\sqrt{N}$ convergence of the standard and the SRD-based MC methods. Moreover, we find
that SRD-MC requires up to 4 times
fewer samples than the standard MC method for the same accuracy. This
difference is more pronounced in the right figure, which is for a probability \eqref{eq:chance} close to 1. This is a first numerical confirmation of \cref{lem:varred1}, that is, the result that SRD-MC outperforms standard MC as $p\to 1$.  In both
figures, spherical-radial QMC significantly outperforms the MC methods, yielding a convergence rate in between $1/\sqrt{N}$ and
$1/N$. For example, the QMC-based estimate achieves an error similar to that with 2000 samples as the other methods with $10^5$ samples (left figure). The figure on the right (i.e., case $p$ is close to 1) shows that the spherical-radial QMC sampling obtains an error similar to that of the SRD MC with 40,000 and the standard MC method with 100,000 samples.

\begin{figure}\centering 
  \begin{tikzpicture}\centering
    \begin{loglogaxis}[title={%
          $p_{\text{true}}\approx 0.6496$},
        xlabel=\# samples N,
        ylabel={RMSE},
        xmin=10, xmax=1e5, ymax=0.2, ymin=3e-5,
        height=7cm, width=7.6cm,
        ticklabel style = {font=\small}]
      \addplot[green!80!black,very thick] table [x=N,y=vanMC] {data/biyupylow03k20_rmse.txt};
      \addplot[blue!80!black,very thick] table [x=N,y=MC] {data/biyupylow03k20_rmse.txt};
      \addplot[red!80!black,very thick] table [x=N,y=QMC] {data/biyupylow03k20_rmse.txt};
      \addplot[dashed,black,thick] coordinates {
        (10,5e-2)
        (100000,5e-4)
      };
      \addplot[dotted,blue,thick] coordinates {
        (10,2e-2)
        (1000,2e-4)
      };
    \end{loglogaxis}
  \end{tikzpicture}
  \begin{tikzpicture}\centering
    \begin{loglogaxis}[title={%
          $p_{\text{true}}\approx 0.9848$},
        xlabel=\# samples N,
        xmax=1e5, xmin=100, ymax=0.2, ymin=3e-5,
        height=7cm, width=6.8cm,
        ticklabel style = {font=\small},
        legend style={legend pos=north east,font=\small} ]
      \addplot[green!80!black,very thick] table [x=N,y=vanMC] {data/biyupylow07k20_rmse.txt};
      \addlegendentry{MC}
      \addplot[blue!80!black,very thick] table [x=N,y=MC] {data/biyupylow07k20_rmse.txt};
      \addlegendentry{sph-rad MC}
      \addplot[red!80!black,very thick] table [x=N,y=QMC] {data/biyupylow07k20_rmse.txt};
      \addlegendentry{sph-rad QMC}
      \addplot[dashed,black,thick] coordinates {
        (10,5e-2)
        (100000,5e-4)
      };
      \addlegendentry{$1/\sqrt{N}$}
      \addplot[dotted,blue,thick] coordinates {
        (10,2e-2)
        (1000,2e-4)
      };
      \addlegendentry{$1/{N}$}
    \end{loglogaxis}
  \end{tikzpicture}  
  \caption{\label{fig:chance-conv}Root mean squared error (RMSE) of probability estimation is shown on $y$-axis for different sampling methods and different number $N$ of MC samples ($x$-axis).  The left plot is for $\underline y\equiv -0.3$, $\bar
  y\equiv 0.3$, corresponding to the probability $p\approx 0.6496$.  The right plot uses the wider bounds $\underline y\equiv -0.7$, $\bar y\equiv 0.7$, resulting in $p\approx 0.9848$. }
\end{figure}
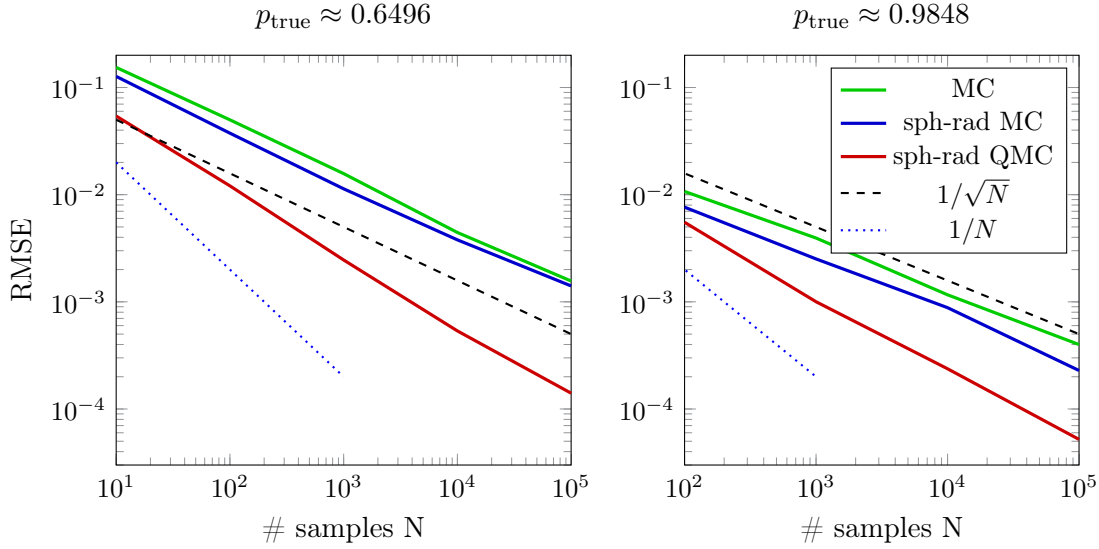

Furthermore, on the left in \cref{fig:chance-conv-KL}, we study the influence of the
number of KL modes (and therefore the dimension) used for the state variable $y$ defined in
\eqref{eq:kl-eta}. To compute a reference solution, we use $K=30$ KL modes and $10^7$ QMC samples. As can be seen, down to an RMSE of about $2\cdot 10^{-3}$ (obtained with 1000 samples), there is little difference between $K=10,15,20$.
The KL truncation error
starts to dominate the overall error at an RMSE of $2\cdot 10^{-3}$ for 10 KL modes and an RMSE of $5\cdot 10^{-4}$ for 20 KL modes.

\begin{figure}\centering
  \begin{tikzpicture}\centering
    \begin{loglogaxis}[
        xlabel=\# samples N,
        ylabel={RMSE},
        xmin=10, xmax=1e5, ymax=0.2, ymin=3e-5,
        height=7cm, width=7cm,
        ticklabel style = {font=\small}]
      \addplot[red!80!black, very thick] table [x=N,y=dim20] {data/biyupylow03kXX_rmse.txt};
      \addlegendentry{$K=20$}
      \addplot[red!80!black, very thick, dotted] table [x=N,y=dim15] {data/biyupylow03kXX_rmse.txt};
      \addlegendentry{$K=15$}
      \addplot[red!80!black, very thick, dashed] table [x=N,y=dim10] {data/biyupylow03kXX_rmse.txt};
      \addlegendentry{$K=10$}
    \end{loglogaxis}
  \end{tikzpicture}\hspace{5mm}
  \begin{tikzpicture}\centering
    \begin{axis}[
        xlabel=probability $p$,
        height=7cm, width=7cm,
        legend pos=south west,
        legend columns = 2,
        scaled ticks = false,
        yticklabel style={
        /pgf/number format/fixed,
        font=\small,
        },
        ]
      \addplot[green!80!black, thick,  mark=*, only marks] table [x
        index=1,y expr={500*\thisrow{var/(1-mean)}}] {data/var-red-vanMC.dat};
      \addlegendentry{$\frac{\VMC}{(1-p)}\:\:\:$}
      \addplot[blue!80!black, thick, mark=square*, only marks] table
              [x index=1,y expr={500*\thisrow{var/(1-mean)}}]
              {data/var-red-sph.dat};
      \addlegendentry{$\frac{\VSRD}{(1-p)}$}
    \end{axis}
  \end{tikzpicture}
  \caption{Left: QMC sampling with SRD for
    different number $K$ of KL modes. Shown is the RMSE of the
    probability estimate for the same setup also used on the left of
    \cref{fig:chance-conv}, i.e., with $\underline y\equiv -0.3$,
    $\bar y\equiv 0.3$. Right: Variance normalized by $(1-p)$ using
    standard and SRD MC samples for the problem described in
    \cref{sec:example-linear}. The difference in $p$ is due to using
    different bounds $\bar y = -\underline y\equiv
    0.5,0.6,0.7,0.8,0.9$ (from left to right, going from more to less
    restrictive). The variance is estimated using 100 MC
    simulations, each using $N=500$ samples.
  \label{fig:chance-conv-KL}}
\end{figure}
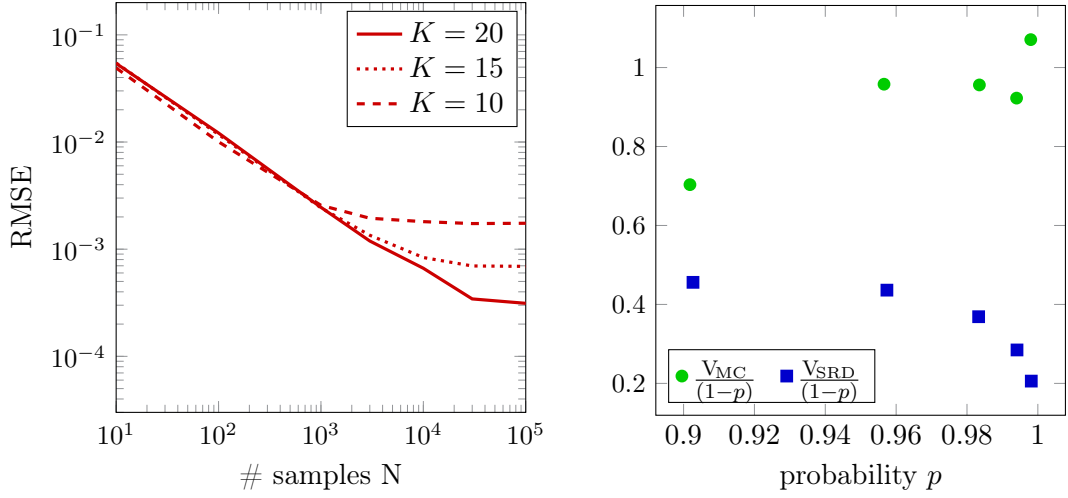

Finally, to verify \cref{lem:varred1} numerically, we
compute the variance of the standard MC estimator
\eqref{eq:stdMC} and the SRD MC estimator
\eqref{srdestimator}. The results are shown on the right in
\cref{fig:chance-conv-KL}, where we divide the numerical
approximation for the variances $\VSRD$ and $\VMC$ defined in
\eqref{SRD-MC} by $(1-p)$ for visualization purposes.
The figure
shows the reduced variance of the SRD estimator, i.e.,
the MC estimator based on the SRD substantially improves over the standard estimator for large $p$. In particular, the ratio
\eqref{eq:SRD/MC} decreases as $p\to 1$.

\subsection{Optimal controls under chance constraints}
To solve the optimal control problem, we employ a sequential quadratic programming (SQP) method, in which we approximate the second derivative of the chance constraint using the BFGS method
\cite{NocedalWright06}. For ease of visualization, we use the previously discussed problem, but only with an upper-state constraint of $\bar y\equiv 0.3$.  In \cref{fig:lin-opt-sol} we show the optimal
control for $p=0.9$ and samples from the state variable. Furthermore, we show the optimal control for $p=0.98$. We observe that the optimal control is quite sensitive to $p$; for $p$ closer to 1 it takes much smaller values to ensure that the states are sufficiently likely to satisfy the bound constraint at all points.

\begin{figure}[tb]\centering
     {\includegraphics[,trim=0
        0 0 0, clip, width=0.47\columnwidth]{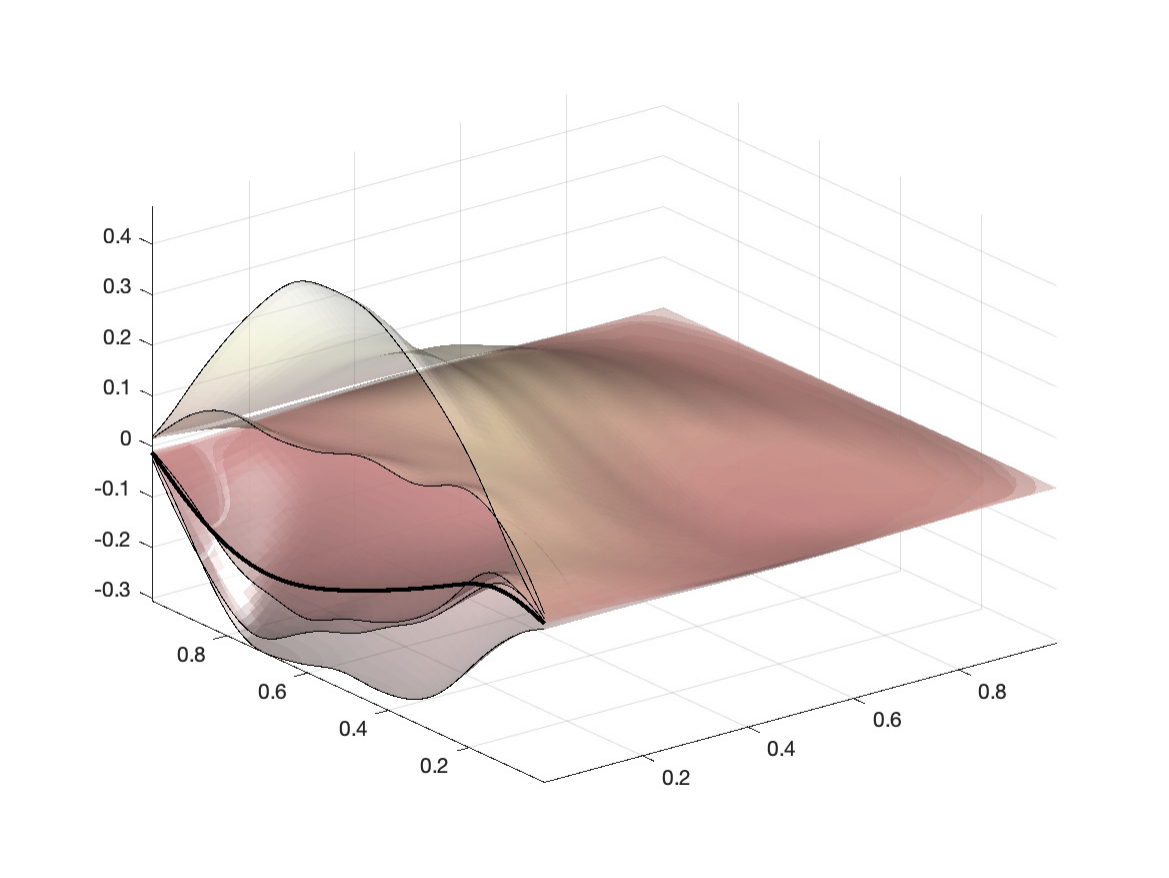}}\hspace{0cm}
        \begin{tikzpicture}
    \begin{axis}[width=.42\columnwidth,
        xmin=0,
        xmax=1,
        ymin=-0.05,
        ymax=0.5,
        compat=1.3,
        ylabel={$\max_{x_1}(y(x_1,x_2))$},
        xlabel={$x_2$},
        legend pos= north east,
        height=5.5cm]
      \addplot[color=blue!50!green,mark=none,thick,dashed] table [x index=0,y
        index=2]{data/Ex1-no-const.dat};
      \addlegendentry{$\bar y$}
      \addplot[color=black!90!white,mark=none,thick] table [x index=0,y
      index=1]{data/Ex1-no-const.dat};
      \addlegendentry{mean}
      \foreach \nn in {5,...,15}
      {
        \addplot[color=black!50!white,mark=none] table
        [x index=0,y index=\nn]{data/Ex1-no-const.dat};
      }
      \addplot[color=blue!70!white,mark=none] table
              [x index=0,y index=8]{data/Ex1-no-const.dat};
      \addplot[color=blue!70!white,mark=none] table
        [x index=0,y index=13]{data/Ex1-no-const.dat};
      \addlegendentry{samples}
      \addplot[color=black!90!white,mark=none,thick] table [x
        index=0,y index=1]{data/Ex1-no-const.dat};
    \end{axis}
  \end{tikzpicture}
  {\includegraphics[width=0.35\textwidth]{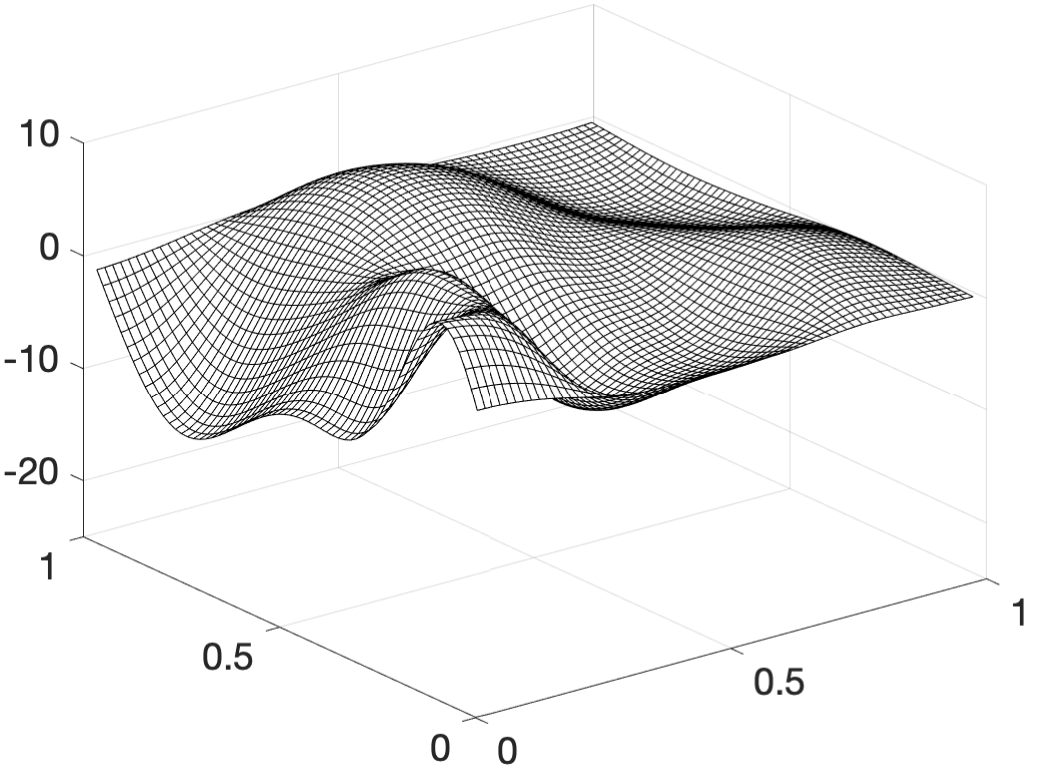}}\hspace{1cm}
 {\includegraphics[width=0.35\textwidth]{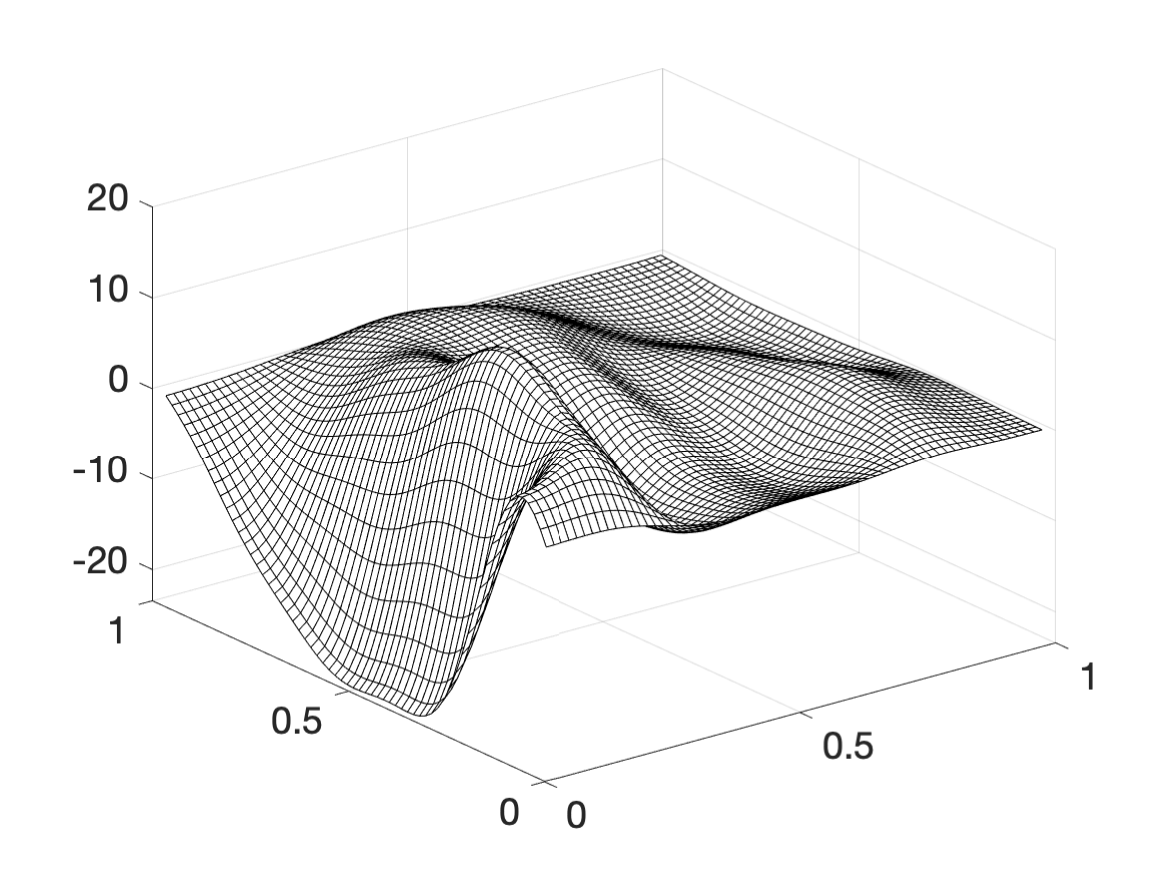}}
  \caption{\label{fig:lin-opt-sol}
    Optimal control solution for $p=0.9$ and unilateral chance constrains with $\bar y=0.3$. Shown on the top left are samples from the state variable $y$. The graph on the top right visualizes $\max_{x_1}(y(x_1,_2))$ as a function of $x_2$ for different samples of the state. The mean of the optimal state is shown in black, samples in gray 
    and blue.  The upper bound $\bar y$ is shown in green (dotted).  The
    two blue samples exceed the bound and thus do not satisfy the bound constraint. Shown at the bottom left is the corresponding optimal control $u$. For comparison, the optimal control $u$ for $p=0.98$ is shown on the bottom right.}
\end{figure}

\section{Numerical results for the bilinear control problem}\label{sec:example-bilinear}
We now extend the results from the previous section to \Cref{ex:bilinear}. Due to the bilinear structure, for a fixed control $u$, the map from the uncertain parameter $\xi$ to the state variable $y$ is linear. Although many methods from the previous section generalize to this case, the dependence of the PDE operator on the control makes the computation of derivatives of the probability function more involved and does not allow for the straightforward dimension reduction for the uncertain variable from \cref{subsec:KL}. 

\subsection{Setup and discretization}
In \Cref{ex:bilinear}, we let $\xi \sim \mathcal N (\xi_0, \mathcal C_0)$, where the covariance operator $\mathcal C_0:L^2(D)\to L^2(D)$ is given by the inverse elliptic operator $(-\alpha\Delta + I)^{-2}$; see~\cite{buithanh_computational_2013}.
 We only consider an upper bound for the state variable, i.e., formally set $\underline y=-\infty$ in \eqref{eq:chance}. 
We use finite elements to approximate the governing equation as well as the prior covariance operator. In particular, we use an $n$-dimensional finite element subspace $V_h\subset H^1_0(\D)$, and write the governing equation \eqref{eq:state_bilin_2} in the form
$[A + M(u)]y(\omega) = M (f+\xi(\omega))$,
where, for $v,g\in V_h$, we define the operators $A,M(u)$ and $M$ through the weak forms
\begin{equation*}
\langle A y, v\rangle := \int_\D \nabla y\cdot \nabla v  \, dx,\quad \langle M(u) y, v\rangle := \int_\D  u y v \, dx, \quad \langle Mg, v \rangle := \int_\D g v \, dx.
\end{equation*}
In the following, we associate each operator with its corresponding matrix, i.e.,
    $A_{ij} = \langle A \phi_i, \phi_j\rangle$,
where $\{\phi_1, \ldots, \phi_n\}$ is a basis for the space $V_h$. As usual in the finite element method, we denote by $\bs u\in \mathbb R^n$ the coefficient vector corresponding to the finite element function $u_h$ (which we simply denote by $u$ again), and similar for the all other variables.  
Expressed in terms of the above operators, the discretized prior covariance is $\mathcal C_0 = ([\alpha A+M]^{-1} M)^2$ and samples of $\bs \xi \sim \mathcal N (\bs \xi_0, \mathcal C_0)$ can then be generated (see~\cite[Sec.~3.6]{buithanh_computational_2013}) via
    $\bs \xi = \bs \xi_0 + [\alpha A+M]^{-1} L\bs z$,
where $\bs z\in \mathbb R^n$ is a draw from an i.i.d.~standard normal distribution, and $L$ satisfies $LL^T=M$.
For brevity and to emphasize that the following analysis does not depend on the choice of the covariance operator,
we denote $\tilde L := [\alpha A+M]^{-1}L$.  Next, we sketch the computation of the probability function and its gradient for this bilinear control problem.

\subsection{Chance constraint probability function and its gradient}
As before, for fixed control $u$ we use the SRD to generate samples of the (finite element coefficient) of the state variable as follows
\begin{equation}\label{eq:bilinear-yi}
    \bs y_i = [A+M(u)]^{-1}M \bs \xi_0 + r_i [A+M(u)]^{-1} M(\bs f+\tilde L\bs v_i)\qquad (i=1,\ldots,N),
\end{equation}
where $r_i$ are random draws from the Chi-distribution, and $\bs v_i$ draws from the uniform distribution of the unit sphere $\mathbb{S}^{n-1}\subset \mathbb R^n$, and $\bs f$ are the coefficients of the discretization of $f$.
We denote $\bs y_{0}^u := [A+M(u)]^{-1}M(\bs f+\bs \xi_0)$ and 
obtain the approximating probability function
\begin{equation}
    \tilde\varphi(u) := {N}^{-1}\sum_{i=1}^N F_{\chi}\big(\rho(u, \bs v_i)\big) \quad \text{ with } \rho(u,\bs v_i) \text{ defined in }\eqref{eq:rho_j}.
\end{equation}
Note that the partial derivatives \eqref{eq:nabla_u} and \eqref{eq:nabla_z}, which are needed in the expressions for $\nabla\tilde\varphi(u)$ given by \eqref{gradform}, involve the term $(A(u)^{-1})'h$. To obtain an explicit form of this derivative, we use the adjoint method to compute the gradient of $\tilde\varphi(u)$, i.e., we consider on the right-hand side in the definition of $\tilde\varphi(u)$ the variables $u, \bs y_i$ and $\bs y_{0}^u$ as independent of each other, but constrain the objective by the equation for the mean state and for the state samples \eqref{eq:bilinear-yi}. Following the (formal) Lagrange method, we compute $\nabla\tilde\varphi(u)$ by taking partial variations of the corresponding Lagrangian function \cite{Troltzsch10}. We find that in addition to solving the equation for the state mean and the $N$ equations in \eqref{eq:bilinear-yi}, the gradient computation requires another $(N+1)$ solutions of adjoint equations (which coincide with the state equation as the PDE operator is self-adjoint).  Using this gradient, one can now solve optimal control problems with chance constraints, as illustrated next for concrete data.

\subsection{Optimal controls under chance constraints}
We consider a problem of the form \cref{ex:bilinear}, without the control bound $u\ge 0$. Using the domain $\D = [0,1]^2$, we defining the auxiliary variable $y_d := \sin(2\pi x_1) \sin(2\pi x_2)$ and choose
     $f:=-\Delta y_d + y_d$, $u_0 :\equiv 1$, $\xi_0 :\equiv 0$.
This data is constructed in such a way that the optimal control is known in the absence of uncertainty. Namely, when the random variable $\xi(\omega)$ reduces to its mean $\xi_0=0$, the optimal control is $u\equiv 1$, and the optimal state is $y=y_d$. For the (inverse) PDE operator that defines the covariance $\mathcal C_0$, we use $\alpha=0.1$ and homogeneous boundary conditions.

\begin{figure}
    \centering
    \includegraphics[width=0.99\textwidth]{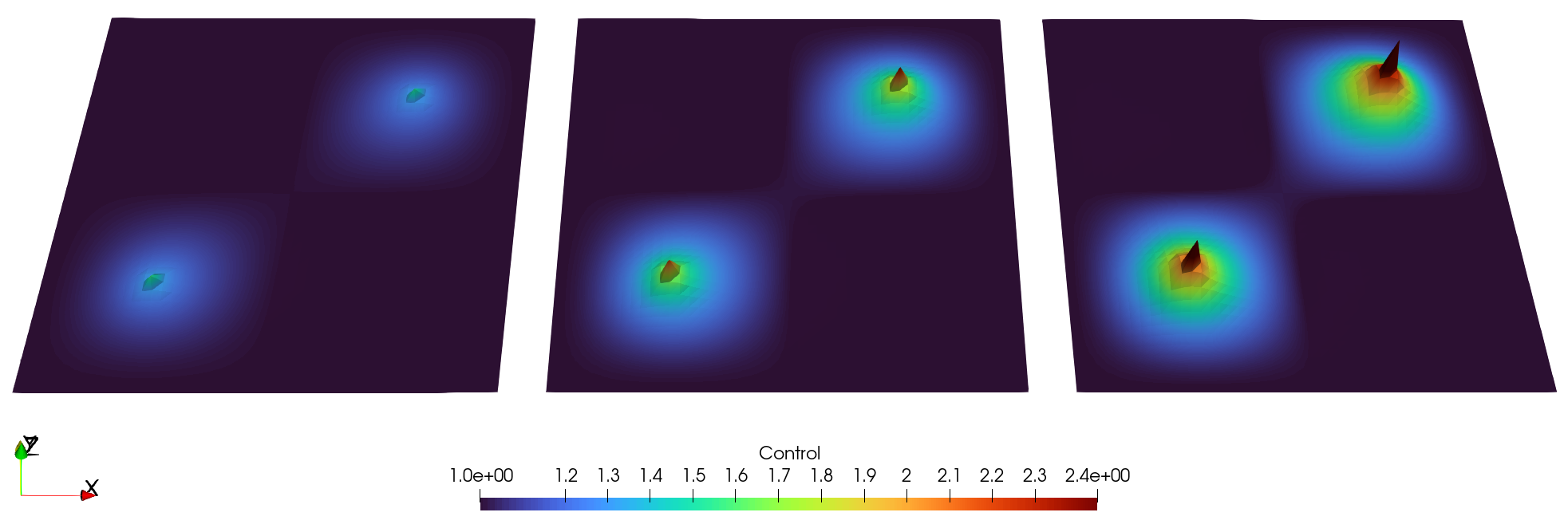}
    \caption{Optimal controls $u$ for $p=0.82, 0.84, 0.86$ (from left to right) for bilinear example. Note that the optimal controls develop spikes near the points where the state distribution is close to the upper bound.}
    \label{fig:bilinear:u}
\end{figure}
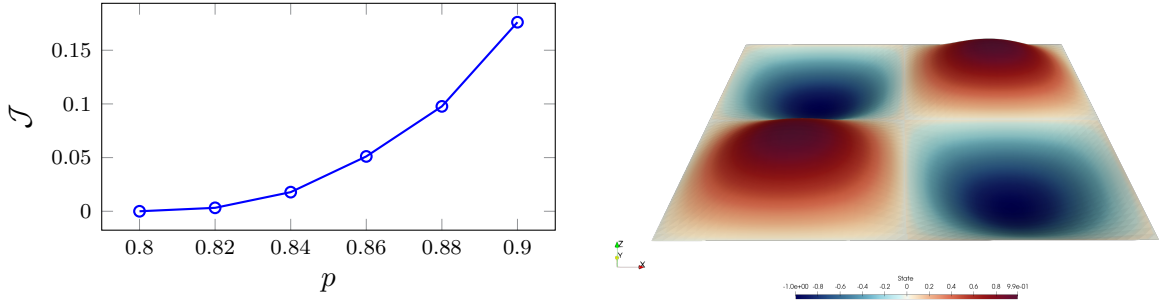
\begin{figure}
    \centering
    \pgfplotstableread[col sep=comma, row sep=\\]{%
p,J\\
0.80,0\\
0.82,3.2090548230041750e-03\\
0.84,1.7815677654847137e-02\\
0.86,5.1092425259438405e-02\\
0.88,9.7708132347274754e-02\\
0.90,1.7610522241182375e-01\\
}\bilineardata
  \begin{tikzpicture}
    \begin{axis}[
  width=6cm,
  height=3.0cm,
  yticklabel style={font=\scriptsize},
  xticklabel style={font=\scriptsize},
  scale only axis,
  yticklabel style={
          /pgf/number format/fixed,
          /pgf/number format/precision=5
  },
  scaled y ticks=false,
  xticklabel style={
          /pgf/number format/fixed,
          /pgf/number format/precision=2
  },
  scaled x ticks=false,
  xlabel={$p$},
  ylabel={${\mathcal J}$}
      ]
      \addplot+[thick, solid, blue, mark=o, mark options={fill=none}, mark size=2pt] table[x=p, y=J] {\bilineardata};
    \end{axis}
  \end{tikzpicture}\hspace{.5cm}
  \includegraphics[width=0.48\textwidth]{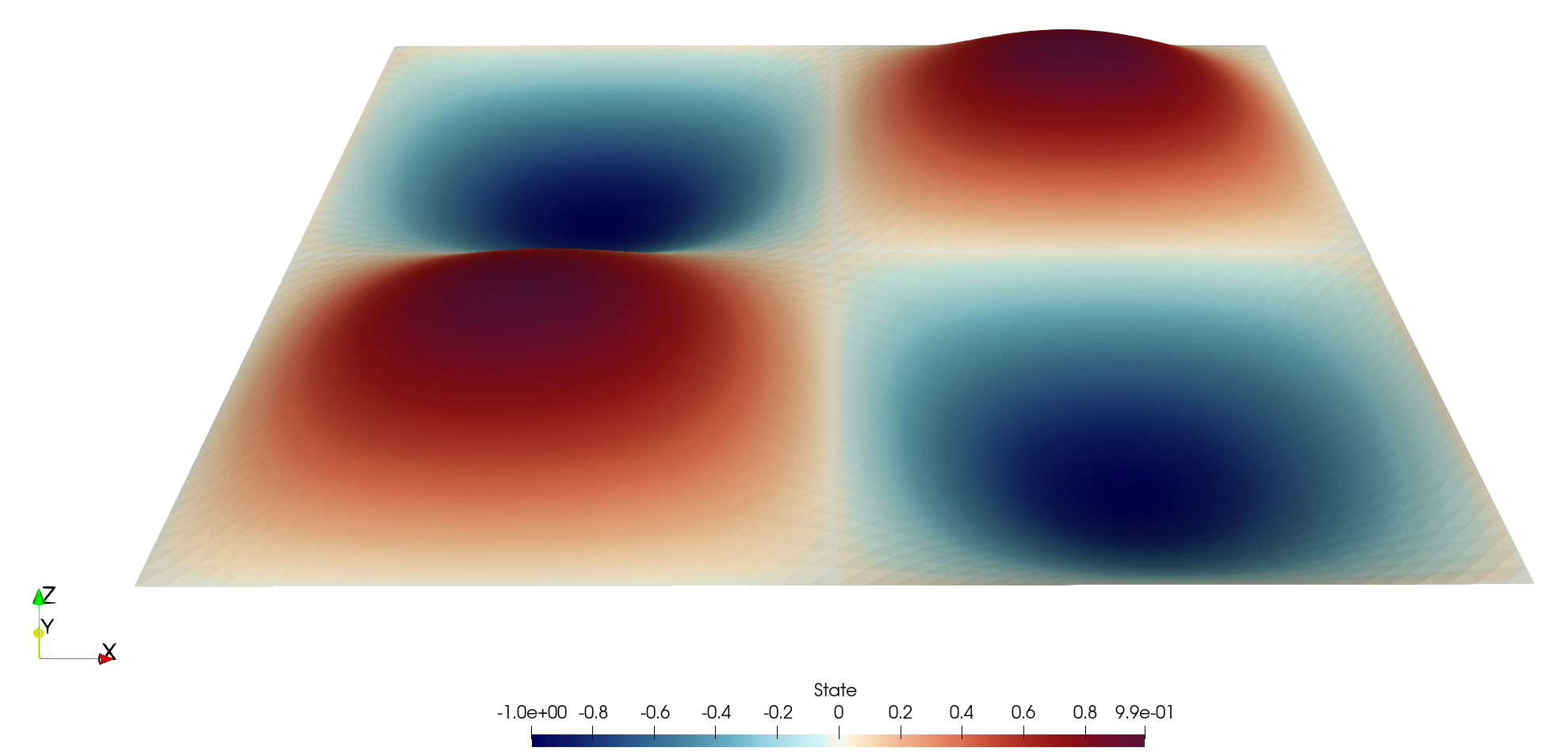}
    \caption{Left: Value of objective $\mathcal J$ at the optimal control $u$ as $p$ in the chance constraint is increased. Right: Mean of state corresponding to optimal control for $p=0.84$ shown in the middle in \cref{fig:bilinear:u}.}
    \label{fig:bilinear:y}
\end{figure}

We choose the upper bound $\bar y = 1.10$ for the states and first solve the optimization problem without the joint chance constraint.
For the resulting optimal control, the probability of exceeding $\bar y$ is $19.6\%$. We then add the chance constraint and solve the optimization problem for $p=0.82,0.84,0.86$, i.e., enforce an exceedance probability of less than $(1-p)$. The resulting optimal controls $u$ are shown in \cref{fig:bilinear:u}, and the objective values corresponding to these controls on the left in \cref{fig:bilinear:y}. As can be seen, the objective increases rapidly as $p$ increases. Also shown on the right in \cref{fig:bilinear:y} is the mean of the optimal state corresponding to $p=0.84$.
Note that the optimal controls $u$ are increased, compared to the solution $u\equiv 1$ without chance constraints, at the two points $(1/4,1/4)$ and $(3/4,3/4)$. This is due to the damping effect that the control has on the states, which is necessary when the mean of the state variable is close to the upper bound $\bar y$. The increasingly singular behavior of controls for increasing $p$ is a consequence of the unboundedness of Gaussian distributions, and the analysis of its behavior and its accurate approximation are interesting future research questions.

\section*{Appendix}
The subsequent two lemmas are slight extensions of the results from \cite[Thm.\ 10.2.1]{prekopa1995} and \cite[Lem.\ 2]{FarshbafShakerHenrionHoemberg18}, adapted for infinite-dimensional image spaces. Since a direct derivation from these results is not possible, we provide independent proofs for the readers' convenience.
\begin{lemma}\label{convexitylemma}
Let $U,\tilde{Y}$ be normed spaces, $K\subseteq \tilde{Y}$ be a convex subset, and $g:U\times\mathbb{R}^n\to\tilde{Y}$ be a linear mapping. Suppose further that $\bs \zeta$ is an $n$-dimensional random vector on some probability space $(\Omega, \mathcal{A},\mathbb{P})$ whose distribution has a log-concave density. Then, the set
\[
M:=\{u\in U:\mathbb{P}(g(u,\bs\zeta )\in K)\geq p\}
\]
is convex for arbitrary $p\in [0,1]$.
\end{lemma}
\begin{proof}
Let $p\in [0,1]$ be arbitrary. Define the set-valued mapping 
$H:U\rightrightarrows\mathbb{R}^n$ by
\[
H(u):=\{\bs z\in\mathbb{R}^n:g(u,\bs z)\in K\}\qquad (u\in U).
\]
Note that the images $H(u)$ are convex subsets of $\mathbb{R}^n$ for all $u\in U$ thanks to $K$ being convex and $g$ being linear. Although convex sets in $\mathbb{R}^n$ need not be Borel measurable, they are Lebesgue measurable (see \cite{lang86}) and hence, with the random vector $\bs \zeta$ having a density by assumption, the probability $\mathbb{P}(\bs \zeta\in H(u))$ is well defined for all $u\in U$. Now, let $u^1,u^2\in M$ and $\lambda\in [0,1]$ be arbitrarily given. This means that $\mathbb{P}(\bs \zeta\in H(u^1))\geq p$ and $\mathbb{P}(\bs \zeta\in H(u^2))\geq p$. We have to show that $\lambda u^1+(1-\lambda )u^2\in M$. We claim that 
\begin{equation}\label{convexinclu}
H(\lambda u^1+(1-\lambda )u^2)\supseteq\lambda H(u^1)+(1-\lambda )H(u^2),
\end{equation}
where the right hand side is the set of element-wise sums.
In fact, assuming $\bs z\in\lambda H(u^1)+(1-\lambda )H(u^2)$, there must exist $\bs z^1\in H(u^1)$ and $\bs z^2\in H(u^2)$ such that $\bs z=\lambda \bs z^1+(1-\lambda )\bs z^2$. Accordingly, $g(u^1,\bs z^1), g(u^2,\bs z^2)\in K$ and, hence, by linearity of $g$ and convexity of $K$:
\[
g(\lambda u^1+(1-\lambda )u^2,\bs z))=g(\lambda (u^1,\bs z^1)+(1-\lambda )(u^2,\bs z^2))=\lambda g(u^1,\bs z^1)+(1-\lambda )g(u^2,\bs z^2)\in K.
\]
This proves that $\bs z\in H(\lambda u^1+(1-\lambda )u^2)$. By \eqref{convexinclu}, we continue as
\begin{eqnarray*}
\mathbb{P}(\bs \zeta\in H(\lambda u^1+(1-\lambda )u^2))&\geq&\mathbb{P}(\bs \zeta\in\lambda H(u^1)+(1-\lambda )H(u^2))\\&\geq& [\mathbb{P}(\bs \zeta\in H(u^1))]^\lambda\cdot [\mathbb{P}(\bs \zeta\in H(u^2))]^{1-\lambda }\\&\geq&p^\lambda p^{1-\lambda}= p
\end{eqnarray*}
Here, the second inequality follows from Pr{\'e}kopas Theorem stating that a log-concave density (assumed here) induces a log-concave probability measure \cite[Theorem 4.2.1]{prekopa1995}. Thus, we have shown the desired relation $\lambda u^1+(1-\lambda )u^2\in M$.
\end{proof}
\begin{lemma}\label{wsus}
Let $U,\tilde{Y}$ be Banach spaces, $K\subseteq \tilde{Y}$ a weakly closed subset and $g:U\times\mathbb{R}^n\to\tilde{Y}$ a weakly sequentially continuous mapping. Suppose further that $\bs \zeta$ is an $n$-dimensional random vector on a probability space $(\Omega, \mathcal{A},\mathbb{P})$. Then, the probability function $\tilde{\varphi}:U\to [0,1]$ defined by
\[
\tilde{\varphi}(u):=\mathbb{P}(g(u,\bs \zeta)\in K)\qquad (u\in U)
\]
is weakly sequentially upper semicontinuous.
\end{lemma}
\begin{proof}
Observe first that  
$\{\omega\in\Omega\mid g(u,\bs \zeta(\omega))\in K\}\in\mathcal{A}$ for arbitrarily $u\in U$
because the sets $\{\bs z\in\mathbb{R}^n\mid g(u,\bs z)\in K\}$ are closed by weak sequential continuity of $g$ and weak closedness of $K$.
Fix an arbitrary $\bar{u}$ and let $u_{n}\rightharpoonup \bar{u}$ be a weakly convergent sequence.
Denote by $u_{n_{l}}$ a subsequence such that 
\begin{equation}
\lim \sup_{n\rightarrow \infty }\tilde{\varphi} (u_{n})=\lim_{l\rightarrow \infty
}\tilde{\varphi} (u_{n_{l}}). \label{realseq}
\end{equation}%
Define the sets $A,A_n\in\mathcal A$ by
\begin{equation*}
A:=\left\{ \omega \in \Omega :g\left( \bar{u},\bs \zeta\left( \omega \right)
\right) \in K\right\} ;\quad A_{n}:=\left\{ \omega \in \Omega :g\left(
u_{n},\bs \zeta\left( \omega \right) \right) \in K\right\} \quad \left( n\in 
\mathbb{N}\right) .
\end{equation*}
Now, consider an arbitrary $\omega \in
\Omega \setminus A$. Then, $g\left( \bar{u},\bs \zeta\left( \omega \right)\right)\notin K$. Since $g$ is weakly sequentially continuous, we have $g\left( u_n,\bs \zeta\left( \omega \right)\right)\rightharpoonup g\left( \bar{u},\bs \zeta\left( \omega \right)\right)$. Since $\tilde{Y}\setminus K$ is weakly open, it follows that
\[
\forall\omega \in\Omega \setminus A\quad\exists n_{0}(\omega ):g(u_{n},\bs \zeta\left( \omega \right) )\notin K\quad \forall n\geq n_{0}(\omega ).
\]
Denoting by $\chi _{C}$
the characteristic function of a set $C$, it follows that
$\chi_{A_{n}}\left( \omega \right) \rightarrow 
0$ as $n\to\infty$ for all $
\omega \in \Omega \backslash A$. By the dominated convergence theorem, 
\begin{equation*}
\int_{\Omega \backslash A}\chi _{A_{n}}\left( \omega \right) \mathbb{P}
\left( d\omega \right) \rightarrow 0\quad \forall
\omega \in \Omega \backslash A \text{ as } n\rightarrow \infty.
\end{equation*}%
On the other hand, $\chi _{A_{n}}\left( \omega \right) \leq \chi _{A}\left(
\omega \right) =1$ for $\omega \in A$, whence%
\begin{eqnarray*}
\lim_{l\rightarrow \infty }\tilde{\varphi} (u_{n_{l}}) &=&\lim_{l\rightarrow \infty }%
\mathbb{P}\left( g\left( u_{n_{l}},\bs \zeta\right) \in K\right)
=\lim_{l\rightarrow \infty }\int_{\Omega }\chi _{A_{n_{l}}}\left( \omega
\right) \mathbb{P}\left( d\omega \right)\\
&\leq &\lim_{l\rightarrow \infty }\sup \int_{\Omega \backslash A}\chi_{A_{n_{l}}}\left( \omega \right) 
\mathbb{P}\left( d\omega \right)
+\lim_{l\rightarrow \infty }\sup \int_{A}\chi _{A_{n_{l}}}\left( \omega
\right) \mathbb{P}\left( d\omega \right) \\
&=&\lim_{l\rightarrow \infty }\sup \int_{A}\chi _{A_{n_{l}}}\left( \omega
\right) \mathbb{P}\left( d\omega \right)\\&\leq&\lim_{l\rightarrow \infty }\sup \int_{A}\mathbb{P}\left( d\omega\right) =\mathbb{P}\left( A\right) =\mathbb{P}\left( g\left( \bar{u},\bs \zeta\right)\in K\right)=\tilde{\varphi} (\bar{u}).
\end{eqnarray*}
Combining this with \eqref{realseq} results in $\varphi $ being weakly sequentially upper semicontinuous in $\bar{u}$.
\end{proof}

\section*{Acknowledgement}
The authors thank two anonymous referees for their valuable comments and suggestions that improved this article.

\bibliographystyle{siamplain}
\bibliography{chance-refs}
\end{document}